\documentclass[a4paper,12pt,reqno,twoside]{amsart}

\usepackage[utf8]{inputenc} 		
\usepackage[T1]{fontenc} 

\usepackage{amssymb}
\usepackage{amsmath}
\usepackage{graphicx}
\usepackage[colorlinks,citecolor=blue,urlcolor=blue]{hyperref}
\usepackage{cite}
\usepackage[hmargin=2.5cm,vmargin=2.5cm]{geometry}
\usepackage{mathrsfs} 

\newcommand{\be}{\begin{eqnarray}}
\newcommand{\ee}{\end{eqnarray}}

\newcommand{\beq}{\begin{equation}}
\newcommand{\eeq}{\end{equation}}
\newcommand{\beqn}{\begin{equation*}}
\newcommand{\eeqn}{\end{equation*}}

\DeclareMathOperator{\tr}{{tr}}

\DeclareMathOperator{\Var}{\mathrm{Var}}

\newtheorem{thm}{Theorem}[section]
\newtheorem{prop}[thm]{Proposition}
\newtheorem{cor}[thm]{Corollary}
\newtheorem{lem}[thm]{Lemma}
\newtheorem{defn}[thm]{Definition}

\newcommand\cB{{\mathcal B}}

\newcommand\cD{{\mathcal D}}
\newcommand\cE{{\mathcal E}}
\newcommand\cF{{\mathcal F}}

\newcommand\cI{{\mathcal I}}

\newcommand\cL{{\mathcal L}}
\newcommand\cM{{\mathcal M}}
\newcommand\cN{{\mathcal N}}

\newcommand\cT{{\mathcal T}}

\newcommand\bE{{\mathbb E}}

\newcommand\bN{{\mathbb N}}

\newcommand\bP{{\mathbb P}}

\newcommand\bR{{\mathbb R}}
\newcommand\bS{{\mathbb S}}

\newcommand\bZ{{\mathbb Z}}

\keywords{Stein's method, multivariate normal approximation, time-dependent dynamical systems, quasistatic dynamical systems}

\begin{document}
\title[CLTs with a rate for sequences of transformations]{Central limit theorems with a rate of convergence for sequences of transformations}
\author[Olli Hella]{Olli Hella}
\address[Olli Hella]{
Department of Mathematics and Statistics, P.O.\ Box 68, Fin-00014 University of Helsinki, Finland.}
\email{olli.hella@helsinki.fi}

\begin{abstract}
Using Stein's method, we prove an abstract result that yields multivariate central limit theorems with a rate of convergence for time-dependent dynamical systems. As examples we study a model of expanding circle maps and a quasistatic model. In both models we prove multivariate central limit theorems with a rate of convergence.      
\end{abstract}
\maketitle
\subsection*{Acknowledgements}I thank my Ph.D. advisor Mikko Stenlund for many helpful advices and suggestions
he gave me during the research and writing process of this paper.
I also thank the Jane and Aatos Erkko Foundation, and Emil Aaltosen S\"a\"ati\"o for their financial support. 

\section{Introduction}\label{Introducion}
Time-dependent dynamical systems have gathered a lot of interest recently, see for example \cite{AyyerStenlund_2007, AyyerLiveraniStenlund_2009,OttStenlundYoung_2009, Stenlund_2011, StenlundYoungZhang_2013, Stenlund_2014, StenlundSulku_2014, DobbsStenlund_2016,  Stenlund_2015, LeppanenStenlund_2016, ConzeRaugi_2007, Nandori_2012, GuptaOttTorok_2013, MohapatraOtt_2014, Kawan_2015, Kawan_2014, Zhu_etal_2012, ArnouxFisher_2005, Aimino_etal_2015, NicolTorokVaienti_2016, KawanLatushkin_2015, Korepanov_2017, Freitas_2018, Haydn_2017,aimino_rousseau2016, Freitas_2017} and for older papers \cite{KolyadaSnoha_1996, KolyadaMisiurewiczSnoha_1999,Bakhtin_1994}. In this paper we approach time-dependent systems by first providing abstract results estimating the distributions of sums of random vectors and variables. The sum of random variables (vectors) is nearly (multi)normally distributed, when certain decay of correlations properties are satisfied. These conditions are specifically designed so that they can be applied to time-dependent dynamical systems yielding CLTs with a rate of convergence.   

To be more precise, the setting we study in this paper is the following:
Let $(X,\cB,\mu)$ be a probability space and $f\colon X\rightarrow \bR^{d}$ a measurable function, where $d\ge 1$, and let $T_{1},T_{2},...$ be measurable transformations on $(X,\cB)$. We denote $\cT_{k}=T_{k}\circ ...\circ T_{1}$ and define $\cT_{0}=\operatorname{Id}$. We study the problem of approximating the distribution of normalized and centered Birkhoff sum
\beqn
W(N)=\dfrac{1}{\sqrt{N}}\sum_{k=0}^{N-1}f\circ \cT_{k}- \mu(f\circ \cT_{k})
\eeqn
by a normal distribution of $d$ variables. The transformations $T_{i}$, $i=1,2,...$, do not have to preserve the measure $\mu$. This restricts the methods that can be used to prove CLTs for the system in question. Denoting $\bar{f}^{i}=f\circ \cT_{i}- \mu(f\circ \cT_{i})$ we may view $W$ as a normalized sum of random variables $\bar{f}^{i}$.  We show that a method in probability theory, introduced by Stein in \cite{Stein_1972}, can be adapted to this setting.

Stein's method has been researched a lot in probability theory, see \cite{Gotze_1991,RinottRotar_1996,GoldsteinRinott_1996,ChatterjeeMeckes_2008,Meckes_2009,ReinertRollin_2009,Nourdin_etal_2010,Chen_etal_2011,Meckes_2012,Rollin_2013,
Gaunt_2016,Berckmoes_etal_2016}, but has mostly been neglected in theory of dynamical systems. Without obtaining convergence rates, the method is applied to some special cases in \cite{Gordin_1990} and \cite{King_1997}, but to our knowledge the first systematic treatment of Stein's method in the context of dynamical systems was not done until recently in \cite{HellaLeppanenStenlund_2016}. The results of \cite{HellaLeppanenStenlund_2016} were then applied in \cite{Leppanen_2017} for non-uniformly expanding maps. In this paper the results of \cite{HellaLeppanenStenlund_2016} are generalized to be applicable for time-dependent systems.

In Section \ref{results in abstract setting} we state two theorems; one concerning random vectors $W$ and second concerning random variables. The proofs of these results are given in Section \ref{main proofs}. These are then applied in two models demonstrating the usefulness of this method. Applications to time-dependent expanding circle maps and in a quasistatic model introduced in \cite{DobbsStenlund_2016} are stated in the sections \ref{circle maps} and \ref{quasistatic model}, respectively. The results for the applications are proved in Sections \ref{computing rho} and \ref{Proofs of quasistatic model results}.

This paper uses some of the results and proofs in aforementioned papers \cite{HellaLeppanenStenlund_2016} and \cite{DobbsStenlund_2016}. We also make certain improvements to  those results.

\subsection*{Notations and conventions}
Through the paper we reserve the letter $Z$ for a random variable with the standard normal distribution. We write $C=C(x_{1},...,x_{n})$, when $C$ is a constant whose numerical value can be calculated from the variables $x_{1},...,x_{n}$.  

Various norms are used through the paper. For a vector $v\in \bR^{d}$ with components $v_{\alpha}$, $\alpha=1,...,d$, we denote
\beqn
|v|=\max\{|v_{\alpha}|: \alpha=1,...,d\}
\eeqn
and for vector valued functions $\|f\|_{\infty}= \max\{\|f_{\alpha}\|_{\infty}:\alpha=1,...,d \}$.
For a function \(B: \, \bR^d \to \bR^{d'}\), we write \(D^kB\) for the $k$th derivative. We define
\begin{align*}
\Vert D^k B \Vert_{\infty} = \max \{ \|\partial_1^{t_1}\cdots\partial_d^{t_d}B_\alpha\|_{\infty}  :  t_1+\cdots+t_d = k,\, 1\le \alpha\le d'  \}.
\end{align*}
Here $B_\alpha$, $1\le \alpha\le d'$ are the coordinate functions of~$B$.

\section{Results in the abstract setting}\label{results in abstract setting}
Let $(X,\cB,\mu)$ be a probability space and $(f^{i})_{i=0}^{\infty}$ a sequence of random vectors. We also assume that every $\|f^{i}\|_{\infty}$, $i\in\bN_{0} $, have a common upper bound denoted by $\|f\|_{\infty}$. We write $\bar{f}^{i}=f^{i}-\mu(f^{i})$ and given an $N\in \bN_{0}$
\beqn
 W=W(N) =\frac{1}{\sqrt{N}}\sum_{i=0}^{N-1}\bar{f}^{i}.
\eeqn
The covariance matrix of $W(N)$ is denoted by $\Sigma_{N}$, i.e.,
\beqn
\Sigma_{N}= \mu(W^{2})=\frac{1}{N}\sum_{i=0}^{N-1}\sum_{j=0}^{N-1}\mu(\bar{f}^{i}\otimes\bar{f}^{j}). 
\eeqn 
Let $K\in\bN_{0}\cap [0,N-1]$. Then we define 
\beqn
[n]_{K}=\{i\in  \bN: 0\le i\le N-1, |n-i|\le K \}
\eeqn
and
\beqn
W_{n}=\frac{1}{\sqrt{N}}\sum_{i=0,i\notin [n]_{K}}^{N-1}\bar{f}^{i}.
\eeqn

We denote by $\Phi_\Sigma(h)$ the expectation of a function $h:\bR^d\to\bR$ with respect to the $d$-dimensional centered normal distribution $\cN(0,\Sigma)$ with positive definite covariance matrix~$\Sigma\in\bR^{d\times d}$, i.e.,
\beqn
\Phi_\Sigma(h) = \frac{1}{\sqrt{(2\pi)^{d}\det\Sigma}} \int_{\bR^d} e^{-\frac12 w\cdot\Sigma^{-1}w}h(w)\, dw.
\eeqn

The next theorem concerns approximating the distribution of the sum of random vectors by a normal distribution. It is formulated in such a way that it can easily be applied to time-dependent dynamical system: Let $X$ be a state space and $f:X\rightarrow \bR^{d} $ a measurable function, and let $(T_{i})_{i=1}^{\infty}$, $T_{i}:X\rightarrow X$, be a sequence of measurable transformations. Denote $\cT_{i}= T_{i}\circ T_{i-1}\circ ... \circ T_{1}$, when $i\ge 1$, and $\cT_{0}=\operatorname{Id}$. Then 
 simply substituting $f^{i}$ in the theorem by $f\circ \cT_{i}$ yields a result for the centered Birkhoff sums $W=\frac{1}{\sqrt{N}}\sum_{i=0}^{N-1} (f\circ \cT_{i}-\mu(f\circ \cT_{i}))$ on the probability space $(X,\cB,\mu)$.

\begin{thm}\label{thm:main}  Let $(X,\cB,\mu)$ be a probability space and $(f^{i})_{i=0}^{\infty}$ a sequence of random vectors with common upper bound $\|f\|_{\infty}\ge \|f^{i}\|_{\infty}$, for every $i\in \bN_{0}$. Let \(h: \, \bR^d \to \bR\) be three times differentiable with \(\Vert D^k h \Vert_{\infty} < \infty\) for \(1 \le k \le 3\).  Fix integers $N>0$ and $0\le K<N$. Suppose that the following conditions are satisfied:

\begin{itemize}
\item[(A1)]\label{A1} There exist constants $C_2 > 0$ and $C_4 > 0$, and a non-increasing function \(\rho : \, \bN_0 \to \bR_+\) with \(\rho(0) = 1\) and \( \sum_{i=1}^{\infty} i\rho(i) < \infty\), such that for all \(0\le i \le j \le k \le l \le N-1\),
\begin{align*}
|\mu(\bar{f}^{i}_{\alpha}\bar{f}^{j}_{\beta})| \leq C_2\rho(j-i),
\end{align*}
\begin{align*}
&|\mu(\bar{f}^{i}_{\alpha}\bar{f}^{j}_{\beta}\bar{f}^{k}_{\gamma}\bar{f}^{l}_{\delta})| \le C_4\rho(\max\{j-i,l-k\}), \\
&|\mu(\bar{f}^{i}_{\alpha}\bar{f}^{j}_{\beta}\bar{f}^{k}_{\gamma}\bar{f}^{l}_{\delta}) - \mu(\bar{f}^{i}_{\alpha}\bar{f}^{j}_{\beta})\mu(\bar{f}^{k}_{\gamma}\bar{f}^{l}_{\delta})| \le C_4\rho(k-j)
\end{align*}
hold whenever $k\ge 0$; $0\le i\le j \le k \le n < N$; $\alpha,\beta,\gamma,\delta\in\{\alpha',\beta'\}$ and $\alpha',\beta'\in\{1,\dots,d\}$.

\smallskip
\item[(A2)]\label{A2} There exists a function \(\tilde\rho : \, \bN_0 \to \bR_+\) such that
\begin{align*}
|\mu( \bar{f}^n \cdot  \nabla h(v + W_n t) ) | \le \tilde\rho(K)
\end{align*}
holds for all $0\le n \le N-1$, $0\le t\le 1$ and $v\in\bR^d$.\footnote{Recall that $W_n$ depends on $K$.}

\smallskip
\item[(A3)]\label{A3} 
$\Sigma_{N}$
is positive-definite $d\times d$ matrix.
\end{itemize}
Then
\beq
|\mu(h(W)) - \Phi_{\Sigma_{N}}(h)| \le 
C_*\left(\frac{K+1}{\sqrt{N}}+\sum_{i=K+1}^{\infty}\rho(i)\right) + \sqrt N\tilde\rho(K),\label{main thm eq1}
\eeq
where
\beq
C_* = 6d^3\max\{C_2,\sqrt{C_4}\}\left(\Vert f \Vert_{\infty} \Vert D^3 h \Vert_{\infty}+ \Vert D^2 h\Vert_{\infty}\right)\sqrt{\sum_{i=0}^{\infty} (i+1)\rho(i)}\label{main thm eq2}
\eeq

is independent of $N$ and $K$.

\end{thm}

This theorem is similar to Theorem 2.1 in \cite{HellaLeppanenStenlund_2016}. The theorem above can be applied to dynamical systems where transformations are time-dependent.
As a side note, the constant $C_{*}$ in \eqref{main thm eq2} is better than in \cite{HellaLeppanenStenlund_2016}.
 
Let $f^{i}$, $i\in \bN_{0}$, be random variables. Then we denote the variance of $W(N)$ by
\beqn
\sigma^{2}_{N}=\mu(W^2)=\frac{1}{N}\sum_{i=0}^{N-1}\sum_{j=0}^{N-1}\mu(\bar{f}^{i}\bar{f}^{j}).
\eeqn
   
For univariate $f^{i}$ we can improve the result of the previous theorem. Instead of  three times differentiable functions $h$, we can assume that $h$ is less regular, namely $1$-Lipschitz, and still get an upper bound result for $\mu(h(W)) - \Phi_{\sigma^{2}_{N}}(h)$. A downside is that the bound obtained is inversely proportional to the variance $\sigma_{N}^{2}$. To state the result rigorously, we introduce the concept of Wasserstein distance.

Let $X_{1}$ and $X_{2}$ be two random variables in $(X,\cB,\mu)$. Then
the Wasserstein distance between $X_{1}$ and $X_{2}$ is defined as
\beqn
d_\mathscr{W}(X_{1},X_{2}) = \sup_{h\in\mathscr{W}}|\mu(h(X_{1})) - \mu(h(X_{2}))|,
\eeqn

where
\beqn
\mathscr{W} = \{h:\bR\to\bR\,:\,|h(x) - h(y)| \le |x-y|\}
\eeqn
is the class of all $1$-Lipschitz functions. 

\begin{thm}\label{main thm 1d}
 Let $(X,\cB,\mu)$ be a probability space and $(f^{i})_{i=0}^{\infty}$ a sequence of random variables with common upper bound $\|f\|_{\infty}$. Fix integers $N>0$ and $0\le K<N$. Suppose that the following conditions are satisfied.

\begin{itemize}
\item[(B1)]\label{B1} There exist constants \(C_2,C_4\) and a non-increasing function \(\rho : \, \bN \to \bR\) with \(\rho(0) = 1\), such that for all \(0 \le i \le j \le k \le l \le N-1\),
\begin{align*}
|\mu(\bar{f}^{i}\bar{f}^{j})| \leq C_2\rho(j-i),
\end{align*}
\begin{align*}
&|\mu(\bar{f}^{i}\bar{f}^{j}\bar{f}^{k}\bar{f}^{l})| \le C_4\rho(\max\{j-i,l-k\}), \\
&|\mu(\bar{f}^{i}\bar{f}^{j}\bar{f}^{k}\bar{f}^{l}) - \mu(\bar{f}^{i}\bar{f}^{j})\mu(\bar{f}^{k}\bar{f}^{l})| \le C_4\rho(k-j).
\end{align*}
\item[(B2)]\label{B2} There exists a function \(\tilde\rho : \, \bN_0 \to \bR_+\) such that, given a differentiable $A:\bR\to\bR$ with $A'$ absolutely continuous and $\max_{0\le k\le 2}\|A^{(k)}\|_\infty \le 1$,
\begin{align*}
|\mu( \bar{f}^n A(W_{n}) ) | \le \tilde\rho(K)
\end{align*}
holds for all $0\le n < N$.
\item[(B3)]\label{B3}
$\sigma_{N}^{2}>0$.
\end{itemize}
Then the Wasserstein distance $d_\mathscr{W}(W,\sigma_{N}Z)$ is bounded from above by 

\beqn
C_{\#}\!\left(\frac{K+1}{\sqrt{N}} + \sum_{i= K+1}^{\infty}  \rho(i) \right) + C_{\#}'\sqrt N\tilde\rho(K),
\eeqn
where
\beqn
C_{\#} = 12 \max\{\sigma_{N}^{-1},\sigma_{N}^{-2}\} \max\{C_2,\sqrt{C_4}\} (1 + \Vert f \Vert_{\infty}) \sqrt{\sum_{i = 0}^{\infty}(i+1) \rho(i)}
\eeqn
and
\beqn
C_{\#}' = 2\max\{1,\sigma_{N}^{-2}\}
\eeqn
are independent of $K$.
\end{thm}

Note that if $\sigma_{N}=0$, then trivially $d_\mathscr{W}(W,\sigma_{N}Z)=0$.

\section{Application I: time-dependent expanding maps}\label{circle maps}
In this section we present some CLTs in a concrete model of expanding circle maps. They are proved in Section \ref{computing rho} by applying Theorems \ref{thm:main} and \ref{main thm 1d}. We also give a result in the case, where transformations are chosen randomly.

\subsection{The model}
Let $\bS^{1}$ be the state space and let $\cM$ denote the set of $C^{2}$ expanding circle maps $T:\bS^{1}\rightarrow \bS^{1}$ with the following bounds:
\beq
\inf T'= \lambda>1,\qquad \Vert T''\Vert_{\infty}\le A_{*}.\label{eq:deriv bounds}
\eeq
Let $f\colon \bS^{1}\rightarrow\bR^{d}$. We write
$
\operatorname{Lip}(f)=\max\{\operatorname{Lip}(f_{\alpha}):\alpha\in 1,...,d \}
$
and
$
\|f\|_{\operatorname{Lip}}=\|f\|_{\infty}+\operatorname{Lip}(f).
$
From now on we assume that all transformations belong to $\cM$ and that $f$ is Lipschitz continuous, i.e., all the coordinate 
functions are Lipschitz continuous.
Furthermore we assume that the initial probability measure $\mu$ with density $\varrho$ with respect to Lebesgue measure $m$ on $\bS^{1}$ is such that $\log \varrho$ is Lipschitz continuous with constant $L_{0}=\operatorname{Lip}(\log \varrho)$. Notice that this implies that $\varrho=e^{\log \varrho}$ is also Lipschitz continuous and $\varrho\ge c>0$ with some $c\in \bR_{+}$. $W$ is defined as in the 
abstract setting in the previous section, as are $\Sigma_{N}$ and $\sigma^{2}_{N}$. 

The results in this section contain constants $\vartheta, C_{2}, C_{4}$ and $B_{0}$. 
Some exact bounds to their values could be calculated by using the results of section 5 
of \cite{DobbsStenlund_2016}, but it is omitted here. Instead we just state here the most important 
features of those constants. First of all $\vartheta \in \,]0,1[$ measures the decorrelation speed of the 
system and depends only on the 
model constants $\lambda$ and $A_{*}$. It is defined as in Lemma 5.6 of \cite{DobbsStenlund_2016}. In particular $\vartheta\ge\lambda^{-1}$. Constants  $C_{2}>0$ and $C_{4}>0$ depend on $\lambda,A_{*},\|f\|
_{\operatorname{Lip}}$ and Lipschitz constant of $\varrho$, and are introduced in Lemma \ref{rho lem multivariate}. The last constant  $B_{0}= B_{0}(L_{0},\lambda,A_{*})>0$ is defined after Lemma \ref{pw lemma 1}.

Now we are ready to present the first theorem concerning expanding circle maps. 
\begin{thm}\label{multivariate circle theorem}
Let $(T_{i})_{i=1}^{\infty}\subset \cM$ be a sequence of transformations in the model $\cM$. Let $h\colon \bR^{d}\rightarrow \bR$ be three times differentiable with $\|D^{k}h\|_{\infty}< \infty$, $k=1,2,3$. Suppose that $N\ge 16/(1-\vartheta)^{2} $ is such that the matrix $\Sigma_{N}$ is positive definite. Then

\begin{align*}
|\mu(h(W)) - \Phi_{\Sigma_{N}}(h)| \le CN^{-\frac{1}{2}} \log N,
\end{align*}
where

\begin{align*}
C= &\,\frac{30d^3\max\{C_2,\sqrt{C_4}\}\left(\Vert f \Vert_{\infty} \Vert D^3 h \Vert_{\infty}+ \Vert D^2 h\Vert_{\infty}\right)}{(1-\vartheta)^{2}}
\\
&+ 2d^{2} \|D^{2}h\|_{\infty}\dfrac{\| f\|_{\operatorname{Lip}}^{2}}{\vartheta^{-\frac{1}{2}} -\vartheta^{\frac{1}{2}}}
+
4d B_{0} \| f\|_{\operatorname{Lip}}\|Dh\|_{\infty} 
+
\frac{2d\Vert Dh \Vert_{\infty}\| f\|_{\operatorname{Lip}}}{\vartheta^{\frac{1}{2}}}.
\end{align*}

\end{thm}

In addition to the previous theorem, for univariate $f$, the following theorem also holds:
\begin{thm}\label{circle main thm}
Let $(T_{i})_{i=1}
^{\infty}\subset \cM$ be a sequence of transformations 
in the model $\cM$. Let $N\ge 16/(1-\vartheta)^{2} $ and $\sigma_{N}\ge C_{0}N^{-p}$, where 
$C_{0}>0$, $p\ge 0$. Then
\begin{align*}
d_\mathscr{W}(W,\sigma_{N}Z) \le \tilde{C}\max\{1,C_{0}^{-2}\}N^{-\frac{1}{2}+2p} \log N, 
\end{align*}
where 

\beq
\tilde{C}=\frac{60\max\{C_2,\sqrt{C_4}\} (1 + \Vert f \Vert_{\infty}) }{(1-\vartheta)^{2}}+\dfrac{4\| f\|_{\operatorname{Lip}}^{2}}{\vartheta^{-\frac{1}{2}} -\vartheta^{\frac{1}{2}}}
+
8 B_{0} \| f\|_{\operatorname{Lip}}
+
\frac{4\| f\|_{\operatorname{Lip}}}{\vartheta^{\frac{1}{2}}}\label{eq:Ctilde}
\eeq
is independent of $N$.
\end{thm}
In particular, if $\sigma_{N} > C_{0}$ (case $p=0$) for $N\ge 3$, the upper bound becomes
\beqn
 \tilde{C}\max\{1,C_{0}^{-2}\}N^{-\frac{1}{2}} \log N .
\eeqn

If the variance $\sigma_{N}$ decreases fast towards zero, then Theorem \ref{circle main thm} is not useful. However, $d_\mathscr{W}(W,\sigma_{N}Z)\le 2\sigma_{N}$ as is proven in Section \ref{computing rho}. Since this second estimate is stronger, when $\sigma_{N}$ is smaller, we are able to provide the following CLT result which is independent of variance.
\begin{cor}\label{N^1/6 convergence cor}  
Let $(T_{i})_{i=1}^{\infty}$, $T_{i}:\bS^{1}\rightarrow \bS^{1}, i\in \bN$ be a sequence of transformations in the model $\cM$. Then
\beqn
d_{\mathscr{W}}(W,\sigma_{N}Z)\leq \max\{\tilde{C},2\}N^{-\frac{1}{6}}\log N,
\eeqn
for all $N\ge 16/(1-\vartheta)^{2} $, where $\tilde{C}$ is as in \eqref{eq:Ctilde}.
\end{cor}
Finally, in the last result of this subsection we consider the self-normalized version of $W$.

For this purpose
define 
\beq
S_{N}=\sum_{i=0}^{N-1}\bar{f}^{i}=\sqrt{N}W(N)=\sqrt{N}W\label{eq: S_N}
\eeq 
and 
\beqn
s_{N}^{2}=\operatorname{Var}(S_{N})=\operatorname{Var}(\sqrt{N}W)=N\sigma_{N}^{2},
\eeqn
i.e., $S_{N}$ is the Birkhoff sum with variance $s_{N}^{2}$. Notice that if $s_{N}>0$, then $S_{N}/s_{N}$ has a variance $1$ and it is thus $W$ after self-normalization. With these definitions, we have the following corollary to Theorem \ref{circle main thm}:

 \begin{cor}\label{Self-normalized cor}
Let $(T_{i})_{i=1}
^{\infty}\subset \cM$ be a sequence of transformations 
in the model $\cM$. Let $N\ge 16/(1-\vartheta)^{2}  $ and $s_{N}^{2}\ge C_{0}N^{p}$, where 
$C_{0}>0$ and $0\le p\le 1$. Then
 \beqn
 d_{\mathscr{W}}\left(\frac{S_{N}}{s_{N}},Z\right)= \tilde{C}\max\{C_{0}^{-\frac12},C_{0}^{-\frac{3}{2}}\}N^{1-\frac{3p}{2}}\log N.
 \eeqn
 \end{cor}
We make two final remarks. 
First, if the growth of $s_{N}^{2}$ is linear ($p=1$), then the upper bound of Wasserstein distance is of the form $CN^{-1/2}\log N$. Second, if $p>2/3$, then $d_{\mathscr{W}}\left(S_{N}/s_{N},Z\right)\rightarrow 0$, when $N\rightarrow \infty$.   

\subsection{Random dynamical system}
In this subsection we study a model, where each transformation is in a set $\Omega_0$ and a sequence $(T_{\omega_{i}})_{i=1}^{\infty}$ of transformations on $\bS^1$ is drawn randomly from a probability space~$(\Omega,\cF,\bP) = (\Omega_0^{\bZ_+},\cE^{\bZ_+},\bP)$. Here $(\Omega_0,\cE)$ is a measurable space and $\bZ_+ = \{1,2,\dots\}$. An initial probability measure $\mu$ is assumed to satisfy the conditions mentioned in the beginning of this section. We assume the following about the random dynamical system in question:

\medskip

\noindent{\bf Assumption (RDS)}
\\
i) Each $T_{\omega_{i}}\in \Omega_{0}$ is an expanding circle map satisfying the bounds in \eqref{eq:deriv bounds}.
\\
ii) The law $\bP$ is stationary, i.e., the shift $\tau:\Omega\to\Omega:(\tau(\omega))_i = \omega_{i+1}$ preserves $\bP$.
\\
iii)  The random selection process is strong mixing, with $\alpha(n)\le Cn^{-\gamma}$, $\gamma>0$, i.e., 
\beqn
\sup_{i\ge 1}\alpha(\cF_1^i,\cF_{i+n}^\infty) \le Cn^{-\gamma}
\eeqn
 for each $n\ge 1$, where $\cF_1^i$ is a sigma-algebra generated by projections $\pi_1,...,\pi_i$, $\pi_k(\omega)=\omega_k$ and $\cF_{i+n}^\infty$ generated by $\pi_{i+n},\pi_{i+n+1}...$; and
\beqn
\alpha(\cF_1^i,\cF_j^\infty) = \sup_{A\in \cF_1^i,\, B\in \cF_j^\infty}|\bP(AB) - \bP(A)\,\bP(B)|.
\eeqn
iv) The map
\beqn
(\omega,x)\mapsto T_{\omega_n}\circ\dots\circ T_{\omega_1}(x)
\eeqn
is measurable from $\cF\otimes\cB$ to $\cB$ for every $n\in\bN = \{0,1,\dots\}$ with the convention $\varphi(0,\omega,x) = x$.

\medskip

 Define $\sigma_{N}^{2}(\omega)=\sigma_{N}^{2}=
\Var_\mu  W(N)$ and $\sigma^{2}=\lim_{N\to\infty}\bE\sigma_{N}^{2}$, when the limit exists. Here $W$ is defined as in the abstract setting except that it now also has $\omega$-dependence. The next theorem gives a quenched convergence result for $W$ that holds for almost every sequence of transformations. 
 \begin{thm} Assume that (RDS) is satisfied. Then $\sigma>0$ if and only if 
 \beqn
 \sup_{N\ge 1} N\,\bE \mu(W^{2})=\infty.
\eeqn  
  Furthermore if $\sigma> 0$ holds, then for almost every $\omega$
 \beqn
d_\mathscr{W}(W(N),\sigma Z) = 
\begin{cases}
O(N^{-\frac 12} \log^{\frac32+\delta}N), &  \gamma> 1,
 \\
  O(N^{-\frac12+\delta}), & \gamma = 1,
  \\
  O(N^{-\frac{\gamma}{2}} \log^{\frac32+\delta}N), &  0<\gamma<1,
\end{cases}
 \eeqn
where $\sigma^{2}= \sum_{k=0}^\infty (2-\delta_{k0})\lim_{i\to\infty}  \bE [\mu(f_i f_{i+k}) - \mu(f_i)\mu(f_{i+k})]$. 
 \end{thm}

The proof is omitted, but we give the following short outline of it: First of all, Assumption (RDS) together with Lemmas \ref{rho lem multivariate} and \ref{pw lemma 1} are applied to show that
Assumptions (SA1)--(SA4) in [Quenched Stein] are satisfied. The condition for $\sigma>0$ is shown by verifying that Assumption (SA5) in [Quenched Stein] holds for the given system and then using Lemma C.2 (iv)(b) $\&$ (v)(b) in that paper.  

Theorem 4.1 in the same paper is then applied, giving the limit variance and bounds for $|\sigma_{n}^{2}(\omega)-\sigma^{2}|$. Lemma \ref{normal distributions distance} and Theorem \ref{circle main thm} are applied to yield the latter part of the above theorem.

\section{Application II: A quasistatic dynamical system}\label{quasistatic model}
The model that we present in this section is introduced in \cite{DobbsStenlund_2016}. First
we present the following definition from \cite{DobbsStenlund_2016}:
\begin{defn}
Let $X$ be a set and $\cM$ a collection of self-maps $T\colon X \rightarrow X$ equipped with a topology. Consider a triangular array
\beqn
\textbf{T} = \{T_{n,k} \in \cM \colon 0\le k \le n, n\ge 1 \}
\eeqn
of elements of $\cM$. If there exists a piecewise continuous curve $\gamma \colon [0,1]\rightarrow\cM$ such that
\beqn
\lim_{n\rightarrow \infty} T_{n,\lfloor nt\rfloor} = \gamma_{t},\qquad t\in [0,1],
\eeqn
we say that $(\textbf{T},\gamma)$ is a \emph{quasistatic dynamical system (QDS)}. The set $X$ is called the \emph{state phase} and $\cM$ the \emph{system phase} of the QDS.
\end{defn}
\subsection{The model}
We define a following QDS, also introduced in \cite{DobbsStenlund_2016}.
The state space is $\bS^{1}$ and the system space $\cM$ is the same set of transformations on $\bS^{1}$ as in the model of Section \ref{circle maps}.  We define a metric $d_{C^{1}}$ to the set $\cM$ by
\beqn
d_{C^{1}}(T_{1},T_{2})= \sup_{x\in \bS^{1}}d(T_{1}x,T_{2}x)+ \Vert T_{1}'-T_{2}'\Vert_{\infty}
\eeqn
for $T_{1},T_{2}\in \cM$. Here $d$ is the natural metric on $\bS^{1}=\bR/ \bZ$.
We assume that $\gamma:\,[0,1]\rightarrow \cM$ is a Hölder continuous curve with exponent $\eta\in \,]0,1[$ and constant $C_{H}\ge 0$. Let $\textbf{T}$ be a triangular array of maps
\beqn
\textbf{T}=\{T_{n,k}\in \cM: 0\le k\le n,n\ge 1\},
\eeqn
which satisfies
\beqn
 \sup_{0\le t\le 1} d_{C^{1}}(T_{n,\lfloor nt\rfloor},\gamma_{t})\le C_{H}n^{-\eta}.
\eeqn
It is known that for every $T\in \cM$ there exists a unique invariant probability measure 
that is equivalent to the Lebesgue measure $m$ on $\bS^{1}$. For $\gamma_{t}$ we denote this 
measure by~$\hat{\mu}_{t}$. Furthermore we write 
$\hat{f}_{t}=f-\hat{\mu}_{t}(f)$.

If $f$ is univariate we define
\beqn
\hat{\sigma}^{2}_{t}(f)= \lim_{m\rightarrow\infty}\hat{\mu}_{t}\left[ \left(\dfrac{1}{\sqrt{m}}\sum_{k=0}^{m-1}\hat{f}_{t}\circ \gamma_{t}^{k}  \right)^{2}\right]
\eeqn
and
\beqn
\sigma^{2}_{t}(f)=\int_{0}^{t}\hat{\sigma}^{2}_{s}(f)ds.
\eeqn
We may write $\sigma^{2}_{t}$ instead of $\sigma^{2}_{t}(f)$ if $f$ is known from the context.

Analogously if $f$ is multivariate we define
\beqn
\hat{\Sigma}_{t}(f)= \lim_{m\rightarrow\infty}\hat{\mu}_{t}\left[ \left(\dfrac{1}{\sqrt{m}}\sum_{k=0}^{m-1}\hat{f}_{t}\circ \gamma_{t}^{k} \right)\otimes \left( \dfrac{1}{\sqrt{m}} \sum_{k=0}^{m-1}\hat{f}_{t}\circ \gamma_{t}^{k}     \right)\right]
\eeqn
and
\beqn
\Sigma_{t}(f)=\int_{0}^{t}\hat{\Sigma}_{s}(f)ds.
\eeqn

Let us introduce some notations. We denote $\cT_{n,i}=T_{n,i}\circ T_{n,i-1}
\circ ... \circ T_{n,1}$ and $\cT_{n,i,j}=T_{n,i}\circ T_{n,i-1}\circ ... \circ T_{n,j}$. 
Furthermore we denote $f_{n,i}= f\circ \cT_{n,i}$ and $\bar{f}_{n,i}= f_{n,i}-
\mu(f_{n,i})$, and define
\beq
\xi_{n}(t)= \xi(t)=\frac{1}{\sqrt{n}}\sum_{i=0}^{\lfloor nt\rfloor-1 } \bar{f}_{n,i}+ 
\dfrac{\{nt\}}{\sqrt{n}}\bar{f}_{n,\lfloor nt\rfloor},\label{xi def}
\eeq
where $\{nt\}=nt-\lfloor nt \rfloor$. Note that
$
\xi_{n}(t)=n^{\frac12}\int_{0}^{t} \bar{f}_{n,\lfloor ns\rfloor} ds.
$

We denote the covariance matrix of $\xi_{n}(t)$ (with respect to $\mu$) by $\Sigma_{n,t}$. If $f$ is univariate, then the variance of $\xi_{n}(t)$ is denoted by $\sigma^{2}_{n,t}$. 
  We aim to prove an upper bound on $|\sigma^{2}_{n,t}-\sigma^{2}_{t}|$ as a function of $n$. This in turn is used to prove an upper bound to Wasserstein distance between  $\xi_{n}(t)$ and $\sigma_{t}Z$.
\subsection{Results}
The next theorem concerns approximating the distribution of $\xi_{n}(t)$ by the multivariate normal distribution $\cN(0,\Sigma_{t})$. By the definition of $\xi_{n}(t)$ 
and Theorem \ref{multivariate circle theorem} it is not surprising that for large $nt$ the distribution of $\xi_{n}(t)$ is 
close to $\cN(0,\Sigma_{n,t})$. Thus the essential new content of this theorem is that $\Sigma_{n,t}\approx \Sigma_{t}$ for large $n$. We also see that the more regular the curve $\gamma$ is, the better is the speed of convergence. 
\begin{thm}\label{multivariate quasitheorem}
Let $t_{0}\in \,]0,1]$ be such that $\hat{\Sigma}_{t_{0}}$ is positive definite and let \(h: \, \bR^d \to \bR\) be three times differentiable with \(\Vert D^k h \Vert_{\infty} < \infty\) for \(1 \le k \le 3\). Then for all $\eta'< 
\eta$ there exists constant $C$ independent of $t$ such that for every $t\ge t_{0}$ and $n\ge 1$
\beqn
\left|\mu(h(\xi_{n}(t)))-\Phi_{\Sigma_{t}}(h)\right|
 \le Cn^{-\eta'}+Cn^{-\frac{1}{2}}\log n.
\eeqn
\end{thm}
It is actually true that if $\hat{\Sigma}_{0}$ is positive definite, then $\hat{\Sigma}_{t_{0}}$ is positive definite with all small enough $t_{0}> 0$. However the constant $C$ depends on the choice of $t_{0}$, which explains the formulation of the previous theorem. 

If $f$ is univariate we can again use the Wasserstein distance. As in the previous theorem, regularity of $\gamma$ effects to the speed of convergence. A simple assumption that $\hat{\sigma}_{t}^{2
}$ is non-zero somewhere is also required for providing the speed of convergence given in the theorem.
\begin{thm}\label{univariate quasitheorem}
Let $t_{0}\in \,]0,1]$ be such that $\hat{\sigma}^{2}_{t_{0}}> 0$. Then for all $\eta'< 
\eta$ there exist constants $C=C(\lambda, A_*,\eta',\eta, C_{H}, \|f\|_{\operatorname{Lip}}, L_{0},t_{0},\hat{\sigma}_{t_{0}}^{2})$ such that for every $t\ge t_{0}$ and $n\ge 1$
\beqn
d_{\mathscr{W}}\left(\xi_{n}(t), \sigma_{t}Z\right) 
 \le Cn^{-\eta'}+Cn^{-\frac{1}{2}}\log n.
\eeqn
\end{thm}
We point out that $\hat{\sigma}^{2}_{t}(f)=0$ only in the very special case that $f=g-g\circ \gamma_{t}$ for some Hölder continuous $g$. 

The last result we present in this section is analogous to Corollary \ref{N^1/6 convergence cor} in Section \ref{circle maps}. It holds without any restriction on the behaviour of the variance $\hat{\sigma}^{2}_{t}$.
\begin{thm} \label{univariate quasitheorem 2}
Let $\eta'< \eta$. Then there exists a constant $C=C(\lambda, A_*,\eta',\eta, C_{H}, \|f\|_{\operatorname{Lip}}, L_{0})$ such that the following holds for every $t\in [0,1]$ and $n\ge 1$:
\beqn
d_{\mathscr{W}}\left(\xi_{n}(t), \sigma_{t}Z\right) 
 \le Cn^{-\frac{\eta'}{2}}+Cn^{-\frac{1}{6}}\log n.
\eeqn
\end{thm}

\section{Proofs for Application I}\label{computing rho}
In this section we study the model described in Section \ref{circle maps}.

\subsection{Upper bounds for $\rho$ and $\hat{\rho}$}
In this subsection we calculate upper bounds for $\rho(K)$ in Assumptions (A1) and (B1), and $\tilde{\rho}(K)$ in Assumptions (A2) and (B2) of Theorems \ref{thm:main} and \ref{main thm 1d}, respectively. 
First we introduce the following definition from \cite{DobbsStenlund_2016}:

\begin{defn}\label{L defn}
We define a class  $\cD_{L}$, $L\in \bR_{+}$, of probability densities $\psi\colon \bS\rightarrow \bR$ in the following way:
$\psi \in \cD_{L}$ if
\\
i) $\psi>0$
\\
ii) there exists $z\in \bS^{1}$ such that $\log \psi$ is Lipschitz continuous on $J_{z}=\bS\setminus \{z\}$ with constant ~$L$.
\end{defn}
Thus $L$ describes the regularity of probability densities in that class, smaller $L$ meaning more regular density. By Remark 4.4iii) in \cite{DobbsStenlund_2016} every Lipschitz continuous probability density $\psi>0$ belongs to $\cD_{L}$ with some value of $L$.

Given a transformation $T\in \cM$, the transfer operator $\cL_{T}\colon L^{1}(m)\rightarrow L^{1}(m)$ is defined by
\beq
\cL_{T}g(x)= \sum_{y\in T^{-1}\{x\}}\frac{g(y)}{T'(y)}.
\eeq
It satisfies the following rule: For every $g\in L^{1}(m)$ and $f\in L^{\infty}(m)$
\beq
\int_{\bS^{1}}f \cL_{T}g dm=\int_{\bS^{1}}g f\circ Tdm.\label{transfer op rule}
\eeq 

Furthermore, we introduce the new notation
$
\cT_{k,j}=T_{k}\circ ...\circ T_{j}.
$

 Applying \eqref{transfer op rule} repeatedly gives 
 $\cL_{\cT_{k,j}}=\cL_{T_{k}}...\cL_{T_{j}}$. We write $\cL_{k,j}=\cL_{T_{k}}...\cL_{T_{j}}$ and $\cL_{k}= \cL_{T_{k}}...\cL_{T_{1}}$.

In general transfer operators tend to smooth probability densities; see, e.g., Lemma 5.2 in \cite{DobbsStenlund_2016}. Concerning this paper, the most important content of that lemma is that there exists a constant 
$L_{*}=L_{*}(\lambda, A_{*})$ with the property that for every $L>L^{*}$ there exists $k$ such that 
\beq
\cL_{k}\cD_{L}\subset \cD_{L_{*}}.\label{L* property}
\eeq
 Actually we can choose
$
L_{*}=A_{*}\lambda(1-\lambda^{-1})^{-2},
$
as the reader may verify by going through the proofs of Lemmas 5.1 and 5.2 in that paper.

Throughout the paper $\vartheta=\vartheta(\lambda,A_{*})\in \,]0,1[$, mentioned in the next lemma, is the same constant as in Lemma 5.6 of \cite{DobbsStenlund_2016}.

The next lemma is implied by Lemma 5.10 of \cite{DobbsStenlund_2016}. Recall that 
$\varrho\in \cD_{L_{0}}$, where $L_{0}=\operatorname{Lip}(\log \varrho)$. By Remark 4.4(ii) of \cite{DobbsStenlund_2016} $L_{0}$ determines some upper bound to $\operatorname{Lip}(\varrho)$, which is why we can replace $\operatorname{Lip}(\varrho)$ by $L_{0}$ in the following lemma:
\begin{lem}\label{rho lem multivariate}
 There exist constants 
\beqn
 C_{2}=C_{2}(\lambda,A_{*},\|f\|_{\operatorname{Lip}},L_{0})>0 \qquad\text{and} \qquad C_{4}=C_{4}(\lambda, A_{*}, \|f\|_{\operatorname{Lip}}, L_{0})>0
\eeqn 
  such that by choosing $\rho(i)=\vartheta^{i}$ the system satisfies the condition \text{(A1)} of Theorem \ref{thm:main}.
and the condition \text{(B1)} of Theorem \ref{main thm 1d}.

\end{lem}
Assume that $0\le n\le N-1$. The following two theorems determine some upper bounds on $|\mu( \bar{f}^n \cdot  \nabla h(v + W_{n} t) ) |$ and $\left|\mu\left[\bar{f}^nA(W_n)\right]\right|$ in the case of multivariate and univariate $f$, respectively. The proof of Theorem \ref{f nabla W lemma} is given after Lemmas \ref{constant approx multivariate}--\ref{splitting lemma}. Proof of Theorem \ref{fAW lemma} is omitted since it follows exactly the same steps as the proof of Theorem \ref{f nabla W lemma}.

 There is some $N$-dependence in the formulations of these theorems which will be removed later to bound $\hat{\rho}(K)$. Therefore only $K$-dependence is left in the formulation of Assumptions (A2) and (B2) for the sake of simplicity. 

\begin{thm}\label{f nabla W lemma}
Given a two times differentiable $h:\bR^{d}\to\bR$ with $\|Dh\|_{\infty},\|D^{2}h\|_{\infty}< \infty$,
\begin{align*}
|\mu( \bar{f}^n \cdot  \nabla h(v + W_{n} t) ) | \le \,&2d^{2} \|f\|_{\infty}\|D^{2}h\|_{\infty}\dfrac{\operatorname{Lip}(f) \lambda^{-\frac{K-1}{2}}}{(\lambda -1)\sqrt{N}}
\\
&+ 4d B_{0}\|f\|_{\infty}\|Dh\|_{\infty} \vartheta^{\frac{K}{2}}
+
2d\Vert Dh \Vert_{\infty}\operatorname{Lip}(f)\lambda^{-\frac{K-1}{2}} 
\end{align*}
holds for all $0\le n \le N-1$, where $B_{0}=C(\lambda,A_{*},L_{0})$ is the constant that will be introduced after Lemma \ref{pw lemma 1}.
\end{thm}

\begin{thm}\label{fAW lemma}
Given a differentiable $A:\bR\to\bR$ with $A'$ absolutely continuous and $\max_{0\le k\le 2}\|A^{(k)}\|_\infty \le 1$,
\begin{align*}
\left|\mu\left[\bar{f}_nA(W_n)\right]\right|
&\leq 2\Vert f\Vert_{\infty} \frac{\operatorname{Lip}(f)\lambda^{-\frac{K-1}{2}}}{(\lambda-1)\sqrt{N}} 
+
4\Vert f\Vert_{\infty}B_{0}\vartheta^{\frac{K}{2}}
+
2\operatorname{Lip}(f)\lambda^{-\frac{K-1}{2}},
\end{align*}
holds for all $0\le n \le N-1$, where $B_{0}=B_{0}(\lambda,A_{*},L_{0})$.
\end{thm}

By Theorem \ref{f nabla W lemma} and inequalities $\lambda^{-1}\le \vartheta$, $\sqrt{N}\ge 1$ and $\|f\|_{\infty},\operatorname{Lip}(f) \le \| f\|_{\operatorname{Lip}}$ we may deduce the following
\begin{align*}
&|\mu( \bar{f}^n \cdot  \nabla h(v + W_{n} t) ) |
\\
\le \, 
&2d^{2} \|f\|_{\infty}\|D^{2}h\|_{\infty}\dfrac{\operatorname{Lip}(f) \lambda^{-\frac{K-1}{2}}}{\sqrt{N}(\lambda -1)}
+
4d B_{0}\|f\|_{\infty}\|Dh\|_{\infty} \vartheta^{ \frac{K}{2}}
+
2d\Vert Dh \Vert_{\infty}\operatorname{Lip}(f)\lambda^{-\frac{K-1}{2}} 
\\
\le \,
&2d^{2} \|D^{2}h\|_{\infty}\dfrac{\| f\|_{\operatorname{Lip}}^{2}\vartheta^{ \frac{K-1}{2}}}{\vartheta^{-1} -1}
+
4d B_{0} \| f\|_{\operatorname{Lip}}\|Dh\|_{\infty} \vartheta^{\frac{K}{2}}
+
2d\Vert Dh \Vert_{\infty}\| f\|_{\operatorname{Lip}}\vartheta^{\frac{K-1}{2}}.
\end{align*}

Thus when we apply Theorem \ref{thm:main} in the model of expanding circle maps introduced in Section \ref{circle maps} we may choose 
in Assumption (A2) that
\beq
 \tilde{\rho}(K) =2d^{2} \|D^{2}h\|_{\infty}\dfrac{\| f\|_{\operatorname{Lip}}^{2}\vartheta^{ \frac{K-1}{2}}}{\vartheta^{-1} -1}
+
4d B_{0} \| f\|_{\operatorname{Lip}}\|Dh\|_{\infty} \vartheta^{\frac{K}{2}}
+
2d\Vert Dh \Vert_{\infty}\| f\|_{\operatorname{Lip}}\vartheta^{\frac{K-1}{2}}.\label{multivar rho tilde}
\eeq

By similar computations for an univariate $f$ we may choose 
in Assumption (B2) that
\beq
\tilde{\rho}(K) = 2\dfrac{\| f\|_{\operatorname{Lip}}^{2}\vartheta^{ \frac{K-1}{2}}}{\vartheta^{-1} -1}
+
 4B_{0} \| f\|_{\operatorname{Lip}} \vartheta^{\frac{K}{2}}
+
2\| f\|_{\operatorname{Lip}}\vartheta^{\frac{K-1}{2}}.\label{univar rho tilde}
\eeq

For the rest of the section, we assume that $n$ is fixed and define $B=\max\{\lfloor n-K/2\rfloor,0\}$.

Each $\cT$ induces a finite partition of $\bS^{1}$ into intervals $I_{i}$, $i\in J$ such that $\cT$ maps $\operatorname{int}I_{i}$ diffeomorphically on $\bS^{1}\setminus \{0 \}$. We call $\{I_{i}: i\in J\}$ the partition induced by $\cT$.  
The next lemma shows that when $K$ is large, then
$\sum_{i=0}^{n-K-1} \bar{f}^{i}$
is almost a constant in elements $I_{i}\in\bS^{1}$ of the partition induced by $\cT_{B}$.

\begin{lem}\label{constant approx multivariate}
Let $I_{i}$ be an element of the partition induced by $\cT_{B}$. There exists $C_{i}\in \bR^{d}$ such that 
\beqn
\left|(v + W_{n}(x) t)- \left(C_{i}+ \frac{t}{\sqrt{N}}\sum_{j=n+K+1}^{N-1}f^{j}(x)\right)\right|\le \frac{\operatorname{Lip}(f)\lambda^{-\left\lfloor\frac{K}{2}\right\rfloor}}{\sqrt{N}(\lambda-1)}
\eeqn

for every $x\in I_{i}$
\end{lem}
\begin{proof}
Assume first that $n\le K$. Then $W_{n}=\frac{1}{\sqrt{N}}\sum_{j=n+K+1}^{N-1} \bar{f^{j}}$. Thus choosing 
$
C_{i}= v - \frac{t}{\sqrt{N}}\sum_{j=n+K+1}^{N-1}\mu(f^{j})
$
yields
\beqn
\left|(v + W_{n}(x) t)- \left(C_{i}+ \frac{t}{\sqrt{N}}\sum_{j=n+K+1}^{N-1}f^{j}(x)\right)\right|= 0.
\eeqn

Assume then that $n> K$. Let $x,y\in I_{i}$ and $j\leq B=n-\lfloor K/2\rfloor$. Then $|\cT_{j}(x)-\cT_{j}(y)|\le \lambda^{j-n+\left\lfloor K/2\right\rfloor}$, which implies
$
|f_{\alpha}\circ \cT_{j}(x)- f_{\alpha}\circ \cT_{j}(y)|\le \operatorname{Lip}(f_{\alpha})\lambda^{j-n+\left\lfloor K/2\right\rfloor}.
$

Thus
\beqn
\left|\frac{t}{\sqrt{N}}\sum_{j=0}^{n-K-1}\bar{f}^{j}_{\alpha}(x)-\frac{t}{\sqrt{N}}\sum_{j=0}^{n-K-1}\bar{f}^{j}_{\alpha}(y)\right|
\le t\frac{\operatorname{Lip}(f_{\alpha})}{\sqrt{N}}\sum_{j=0}^{n-K-1} \lambda^{j-n+\left\lfloor \frac{K}{2}\right\rfloor}
\eeqn
\beqn
\le \lambda^{n-K-1-\left\lfloor n-\frac{K}{2}\right\rfloor}\frac{\operatorname{Lip}(f_{\alpha})}{\sqrt{N}}\sum_{j=0}^{\infty} \lambda^{-j} = \frac{\operatorname{Lip}(f_{\alpha})\lambda^{-\left\lfloor\frac{K}{2}\right\rfloor}}{\sqrt{N}(\lambda-1)}.
\eeqn
Therefore there exists $\bar{C}_{i}\in \bR^{d}$ such that 
\beqn
\left|v+\frac{t}{\sqrt{N}}\sum_{j=0}^{n-K-1}\bar{f}^{j} (x)-\bar{C}_{i}\right| \le \frac{\operatorname{Lip}(f)\lambda^{-\left\lfloor\frac{K}{2}\right\rfloor}}{\sqrt{N}(\lambda-1)}
\eeqn
 for every $x\in I_{i}$. Thus there exists a constant $C_{i}=\bar{C}_{i}-\dfrac{t}{\sqrt{N}}\sum_{j=n+K+1}^{N-1} \mu(f^{j})$ such that
\begin{align*}
&\left|(v + W_{n}(x) t)- \left(C_{i}+ \frac{t}{\sqrt{N}}\sum_{j=n+K+1}^{N-1}f^{j}(x)\right)\right|
\\
=&\left|v+\frac{t}{\sqrt{N}}\sum_{j=0}^{n-K-1}\bar{f}^{j} (x)-\left( C_{i}+\frac{t}{\sqrt{N}}\sum_{j=n+K+1}^{N-1}\mu(f^{j})\right)\right|
\\
\le &\,\frac{\operatorname{Lip}(f)\lambda^{-\left\lfloor\frac{K}{2}\right\rfloor}}{\sqrt{N}(\lambda-1)}
\end{align*}
for every $x\in I_{i}$.
\end{proof}

The next standard lemma shows that the transfer operator decreases the distance of two probability measures in the $L^{1}$ norm.

\begin{lem}\label{pw lemma 1}
Let $\cT_{c,a}$, $\cT_{c,b}$ be two compositions of any maps in $\cM$, where  $a\le b\le c$,
and let $\varrho_{1},\varrho_{2}\in \cD_{L}$, where $L\ge L_{*}$. Then there exist constants   $D_{0}=D_{0}(L,\lambda,A_{*})$ and $\vartheta=\vartheta(\lambda,A_{*})\in \,]0,1[$ such that
\beqn
\Vert\cL_{c,a}(\varrho_{1})-\cL_{c,b}(\varrho_{2})\Vert_{L^{1}}\leq D_{0}\vartheta^{c-b+1}.
\eeqn
\end{lem}
\begin{proof}
Lemma 5.2 (i) in \cite{DobbsStenlund_2016} gives that $\cT_{b-1,a}(\varrho_{1})\in \cD_{L}$. Thus applying 
Lemma 5.6 of the same article gives
\beqn
\Vert\cL_{c,a}(\varrho_{1})-\cL_{c,b}(\varrho_{2})\Vert_{L^{1}}
=\Vert\cL_{c,b}(\cL_{b-1,a}(\varrho_{1})-\varrho_{2}))\Vert_{L^{1}}\leq D_{0}\vartheta^{c-b+1}.
\eeqn
That $D_{0}=D_{0}(L,\lambda,A_{*})$ and $\vartheta=\vartheta(\lambda,A_{*})\in \,]0,1[$ follows from Section 5 of \cite{DobbsStenlund_2016}.
\end{proof}
We define the new constant $L_{1}= \max\left\{L_{*},L_{0}\right\}$.
Lemma \ref{pw lemma 1} now implies that
there exists a constant $B_{0}=B_{0}(\lambda,A_{*},L_{1})=B_{0}(\lambda,A_{*},L_{0})$ such that 
\beq
\Vert\cL_{c,a}(\varrho)-\cL_{c,b}(\varrho)\Vert_{L^{1}}\le B_{0}\vartheta^{c-b+1}. \label{B_0 formula}
\eeq
 
The following result is Lemma 5.2(iii) in \cite{DobbsStenlund_2016}.
\begin{lem}
Let $L\ge L_{*}$, $\varrho_{0}\in \cD_{L}$, $m\ge 1$ and $\cT_{m}$ be a composition of $m$ maps in $\cM$. Then for every $I_{i}$, $i\in J$ it holds that  

\beqn
\cL_{m}\left(\frac{\varrho_{0} 1_{I_{i}}}{\mu_{0}(I_{i})} \right)\in \cD_{L}.
\eeqn 
 \label{pw lemma 2}
\end{lem}

Lemmas \ref{pw lemma 1} and \ref{pw lemma 2} yield the following corollary. 
\begin{cor}\label{pushforward cor}
Let $\varrho_{1},\varrho_{2}\in \cD_{L}$, $L\ge L_{*}$, $1<m+1<n$, $
\cT_{m}$ composition of $m$ maps in $\cM$. Then for every $I_{i}$, $i\in J$ it holds 
that
\beqn
\left\| \cL_{n}(\varrho_{1})-\cL_{n,m+1}\left(\cL_{m}\left(\frac{\varrho_{2} 1_{I_{i}}}{\mu_{2}(I_{i})} \right)\right)\right\|_{L^{1}}\leq D_{0}\vartheta^{n-m},
\eeqn
where $D_{0}=D_{0}(L,\lambda,A_{*})$ is the same constant as in Lemma \ref{pw lemma 1}.
\end{cor}
Similarly to Lemma \ref{pw lemma 1}, the previous corollary holds for two probability densities $\varrho_{1},\varrho_{2}$ in the class $\cD_{L_{0}}\subset \cD_{L_{1}}$. The constant $D_{0}$ is $B_{0}(L_{0},\lambda,A_{*})$, where $B_{0}$ is the same constant as in \eqref{B_0 formula}.

The content of the next lemma is exponential decay of pair correlations when any sequence of transformations in $\cM$ is applied.

\begin{lem}\label{splitting lemma}
Let $g,h: \bS^{1}\rightarrow \bR $, where $g$ is Lipschitz, $h$ bounded and $\cT_{m}=T_{m}\circ ...\circ T_{1} $ a composition of maps in $\cM$ and $\varrho_{0}\in \cD_{L}$, where $L\ge L_{*}$. Then 
\beqn
\left|\int_{\bS^{1}} gh\circ \cT_{m}\varrho_{0} dm - \int_{\bS^{1}} g\varrho_{0} dm\int_{\bS^{1}}h\circ \cT_{m}\varrho_{0} dm\right|\le  \Vert h \Vert_{\infty}\left(2\operatorname{Lip}(g)\lambda^{-\left\lfloor \frac{m}{2}\right\rfloor} + \Vert g\Vert_{\infty}D_{0}\vartheta^{\left\lceil \frac{m}{2}\right\rceil}\right),
\eeqn
where $D_{0}=D_{0}(L,\lambda,A_{*})$  is the same constant as in Lemma \ref{pw lemma 1}.
\end{lem}
\begin{proof}
Let $\{I_{i}:i\in J\}$ be the partition induced by $\cT_{\lfloor m/2\rfloor}$. Thus $|I_{i}|\leq \lambda^{-\lfloor m/2\rfloor}$. Define $g_{i}=\dfrac{\int_{I_{i}}g\varrho_{0} dm }{\int_{I_{i}} \varrho_{0} dm}$. We have $g_{i}=g(x_{i})$ for some $x_{i}\in I_{i}$. Thus
\begin{align}
\begin{split}
\int_{\bS^{1}} gh\circ \cT_{m}\varrho_{0} dm &=\sum_{i}\int_{I_{i}} gh\circ \cT_{m}\varrho_{0} dm
\\
&= \sum_{i}\left(g_{i}\int_{I_{i}} h\circ \cT_{m}\varrho_{0} dm +\int_{I_{i}} (g-g_{i})h\circ \cT_{m}\varrho_{0} dm \right)
\\
&=\sum_{i}g_{i}\int_{I_{i}} h\circ \cT_{m}\varrho_{0} dm + E_{1},\label{expr 1}
\end{split}
\end{align}
where $|E_{1}|\le \sum_{i}\operatorname{Lip}(g)\lambda^{-\lfloor m/2\rfloor}\Vert h\Vert_{\infty}\int_{I_{i}}|\varrho_{0}|dm=\operatorname{Lip}(g)\lambda^{-\lfloor m/2\rfloor}\Vert h\Vert_{\infty}$. Furthermore by the properties of the transfer operator:
\begin{align*}
\int_{I_{i}} h\circ \cT_{m}\varrho_{0} dm &=\int_{\bS^{1}}  h\circ \cT_{m,\lfloor m/2\rfloor +1} \circ \cT_{m/2}\varrho_{0} 1_{I_{i}} dm
\\
&=\mu_{0}(I_{i})\int_{\bS^{1}} h \cL_{m,\left\lfloor m/2\right\rfloor+1}\left( \cL_{\left\lfloor m/2\right\rfloor} \left( \frac{\varrho_{0} 1_{I_{i}}}{\mu_{0}1_{I_{i}}}\right)  \right) dm
\\
&= \mu_{0}(I_{i})\left(\int_{\bS^{1}} h \cL_{m}\left(\varrho_{0}\right) dm+ E_{i}\right),
\end{align*}

where $|E_{i}|\le \Vert h\Vert_{\infty}D_{0}\vartheta^{m-\lfloor m/2\rfloor}= \Vert h\Vert_{\infty}D_{0}\vartheta^{\lceil m/2\rceil}$ by Corollary \ref{pushforward cor}.
Thus \eqref{expr 1} is
\begin{align}
\begin{split}
&\sum_{i}\left(g_{i}\mu_{0}(I_{i})\left(\int_{\bS^{1}} h \cL_{m}\left(\varrho_{0}\right) dm+ E_{i}\right)\right) + E_{1}
\\
=&\sum_{i}g_{i}\mu_{0}(I_{i})\int_{\bS^{1}} h \cL_{m}\left(\varrho_{0}\right) dm+ \sum_{i}\left(\mu_{0}(I_{i})g_{i}E_{i}\right) + E_{1}
\\
=&\sum_{i}\int_{\bS^{1}}g_{i}1_{I_{i}}\varrho_{0}dm\int_{\bS^{1}} h \circ\cT_{m}\varrho_{0} dm+ E_{2} + E_{1},\label{expr 2}
\end{split}
\end{align}
where $|E_{2}|\le \Vert g\Vert_{\infty}\Vert h\Vert_{\infty}D_{0}\vartheta^{\lceil m/2\rceil}$. Furthermore we have
\beqn
\left|\sum_{i}\int_{\bS^{1}}g_{i}1_{I_{i}}\varrho_{0}dm- \int_{\bS^{1}}g\varrho_{0}dm\right|=\left|\int_{\bS^{1}}\left(\sum_{i}g_{i}1_{I_{i}}-g\right)\varrho_{0}dm \right|\leq \operatorname{Lip}(g)\lambda^{-\left\lfloor \frac{m}{2}\right\rfloor}.
\eeqn
From which it follows that
\beq
\left|\sum_{i}\left(\int_{\bS^{1}}g_{i}1_{I_{i}}\varrho_{0}dm \int_{\bS^{1}}h\circ \cT_{m}\varrho_{0} dm \right)- \int_{\bS^{1}}g\varrho_{0}dm \int_{\bS^{1}}h\circ \cT_{m}\varrho_{0} dm \right|\le \operatorname{Lip}(g)\lambda^{-\left\lfloor \frac{m}{2}\right\rfloor}\|h\|_{\infty}.\label{expr 3}
\eeq
By \eqref{expr 1}, \eqref{expr 2} and \eqref{expr 3}, we have
\beqn
\int_{\bS^{1}} gh\circ \cT_{m}\varrho_{0} dm=\int_{\bS^{1}}g\varrho_{0}dm\int_{\bS^{1}} h \circ\cT_{m}\varrho_{0} dm+ E_{3} +E_{2} + E_{1},
\eeqn
where $|E_{3}|\le \operatorname{Lip}(g)\lambda^{-\lfloor m/2\rfloor} \Vert h \Vert_{\infty}$.
Thus 
\begin{align*}
&\left|\int_{\bS^{1}} gh\circ \cT_{m}\varrho_{0} dm-\int_{\bS^{1}}g\varrho_{0}dm\int_{\bS^{1}} h \circ\cT_{m}
\varrho_{0} dm\right|
\\
\le\,& 2\operatorname{Lip}(g)\lambda^{-\left\lfloor \frac{m}{2}\right\rfloor} \Vert h \Vert_{\infty}+ \Vert g\Vert_{\infty}\Vert h
\Vert_{\infty}D_{0}\vartheta^{\left\lceil \frac{m}{2}\right\rceil}.
\end{align*}

\end{proof}

\subsection{The proof of Theorem \ref{f nabla W lemma}}
The overall strategy used in the following proof is described in section 7 of \cite{HellaLeppanenStenlund_2016}.

\begin{proof}
\noindent{\bf Step 1.}
In Step 1 we split the measure $\mu$ to a sum of conditional measures on small intervals $I_{i}$. In these intervals $W_{n}=\frac{1}{\sqrt{N}}\sum_{j=0,j\notin [n]_{K}}^{N-1}\bar{f}^{j}$ can be approximated by $C_{i}+\sum_{j=n+K+1}^{N-1}\bar{f}^{j}$ with only small error. Here $C_{i}\in \bR^{d}$ depends on the interval. The choice of intervals $I_{i}$ is delicate. The smaller the intervals are, the smaller is the error made in Step 1. However, for the purposes of computations in Steps 2 and 3 only very specific choices produce small errors. 

Recall that $B=\max\{\lfloor n-K/2\rfloor,0\}$.
Let $\cI=\{I_{i}:i\in J \}$ be the partition of $\bS^{1}$ induced by $\cT_{B}$. We may represent the measure $\mu$ as $\sum_{i} 
\mu_{i} $, where $\mu_{i}(U)=\mu(U\cap I_{i})$, $U\subset 
\bS^{1}$. Thus
\beq \label{eq:ext1}
\mu( \bar{f}^n \cdot  \nabla h(v + W_{n} t) )=\sum_{i}
\int_{I_{i}}\sum_{\alpha}\left(f^n_{\alpha}-\mu(f^{n}
_{\alpha})\right) \cdot  \partial_{\alpha} h(v + W_{n} t)
\varrho dm. 
\eeq
By Lemma \ref{constant approx multivariate} we may write
$
v + W_{n}(x) t= C_{i}+ \dfrac{t}{\sqrt{N}}\sum_{j=n+K+1}^{N-1}f^{j}(x)+ E(x),
$
where $\|E\|_{\infty}\le  \dfrac{\operatorname{Lip}(f)\lambda^{-\left\lfloor\frac{K}{2}\right\rfloor}}{\sqrt{N}(\lambda-1)}$.
Now the right side of \eqref{eq:ext1} equals

\begin{align}
&\sum_{\alpha} \sum_{i}\int_{I_{i}}\left(\left(f^{n}_{\alpha}-\mu(f^{n}_{\alpha})\right) \partial_{\alpha} h\left(C_{i}+ \frac{t}{\sqrt{N}}\sum_{j=n+K+1}^{N-1}f^{j}+E\right)\right) \varrho dm\nonumber
\\
= &\sum_{\alpha} \sum_{i}\int_{I_{i}}\left(\left(f^{n}_{\alpha}-\mu(f^{n}_{\alpha})\right) \partial_{\alpha} h\left(C_{i}+ \frac{t}{\sqrt{N}}\sum_{j=n+K+1}^{N-1}f^{j}\right)\right) \varrho dm\nonumber
\\
\begin{split}
+
&\sum_{\alpha} \sum_{i}\int_{I_{i}}\left(f^{n}_{\alpha}-\mu(f^{n}_{\alpha})\right)
\\
\cdot &\left(\partial_{\alpha} h\left(C_{i}+ \frac{t}{\sqrt{N}}\sum_{j=n+K+1}^{N-1}f^{j}+E\right)- \partial_{\alpha} h\left(C_{i}+ \frac{t}{\sqrt{N}}\sum_{j=n+K+1}^{N-1}f^{j}\right)\right) \varrho dm\label{E' term}
\end{split}
\\
= &\sum_{\alpha}\sum_{i}\int_{I_{i}}\left(\left(f^{n}_{\alpha}-\mu(f^{n}_{\alpha})\right) \partial_{\alpha} h\left(C_{i}+ \frac{t}{\sqrt{N}}\sum_{j=n+K+1}^{N-1}f^{j}\right)\right) \varrho dm +E',\nonumber
\end{align}
where 
\begin{align*}
|E'|=|\eqref{E' term}| 
\le &\,\sum_{\alpha}\|f^{n}_{\alpha}-\mu(f^{n}_{\alpha})\|_{\infty}
\\
&\cdot \int_{\bS^{1}}\left|\partial_{\alpha} h\left(C_{i}+ \frac{t}{\sqrt{N}}\sum_{j=n+K+1}^{N-1}f^{j}+E\right)- \partial_{\alpha} h\left(C_{i}+ \frac{t}{\sqrt{N}}\sum_{j=n+K+1}^{N-1}f^{j}\right) \right|\varrho dm
\\
\le &\,2\|f\|_{\infty}\sum_{\alpha} \sum_{\beta}\|\partial_{\beta}(\partial_{\alpha} h)\|_{\infty} \| E_{\beta}\|_{\infty}
\\
\le &\,2d^{2} \|f\|_{\infty}\|D^{2}h\|_{\infty}\dfrac{\operatorname{Lip}(f) \lambda^{- \lfloor \frac{K}{2}\rfloor}}{\sqrt{N}(\lambda -1)}.
\end{align*}

Thus we now have
\begin{align}\begin{split}
&\bigg|\mu( f^n \cdot  \nabla h(v + W_{n} t) ) 
-\sum_{\alpha} \sum_{i}\int_{I_{i}}\left(\left(f^{n}_{\alpha}-\mu(f^{n}_{\alpha})\right) \partial_{\alpha} h\left(C_{i}+ \frac{t}{\sqrt{N}}\sum_{j=n+K+1}^{N-1}f^{j}\right)\right) \varrho dm \bigg|
\\ 
&\le 
2d^{2} \|f\|_{\infty}\|D^{2}h\|_{\infty}\dfrac{\operatorname{Lip}(f) \lambda^{- \lfloor \frac{K}{2}\rfloor}}{\sqrt{N}(\lambda -1)}. \end{split} \label{STEP 1}
\end{align}

\noindent{\bf Step 2.}
In Step 2 we modify the integral in the previous equation. The trick done in Step 1 enables to write the integral in \eqref{STEP 1} as
\begin{align*}
\int_{\bS^{1}}G\circ \cT_{n}\varrho 1_{I_{i}} dm &= \mu(I_{i})\int_{\bS^{1}}G\circ \cT_{n,B+1}\cL_{B}\left(\frac{\varrho 1_{I_{i}}}{\mu(I_{i})}\right) dm \\
&= \mu(I_{i})\int_{\bS^{1}}G \cL_{n,B+1}\left(\cL_{B}\left(\frac{\varrho 1_{I_{i}}}{\mu(I_{i})}\right)\right)  dm,
\end{align*}
where $G$ is some function on $\bS^{1}$. By Corollary \ref{pushforward cor}, $\cL_{n}(\varrho)\approx \cL_{n,B+1}\left(\cL_{B}\left(\dfrac{\varrho 1_{I_{i}}}{\mu(I_{i})} \right)\right)$. Thus in the Step 3 we are only left to evaluate $\int_{\bS^{1}}G \cL_{n}(\varrho)dm$.
  
Beginning Step 2, notice that we can write
\begin{align}
\begin{split}
&\sum_{\alpha} \sum_{i}\int_{I_{i}}\left(\left(f^{n}_{\alpha}-\mu(f^{n}_{\alpha})\right) \partial_{\alpha} h\left(C_{i}+ \frac{t}{\sqrt{N}}\sum_{j=n+K+1}^{N-1}f^{j}\right)\right) \varrho dm 
\\
=&\sum_{\alpha} \sum_{i}\int_{I_{i}}\left(\left(f_{\alpha}\circ\cT_{n,B+1}-\mu(f^{n}_{\alpha})\right) \partial_{\alpha} h\left(C_{i}+ \frac{t}{\sqrt{N}}\sum_{j=n+K+1}^{N-1}f\circ \cT_{j,B+1}\right)\right)\circ \cT_{B} \varrho dm.\label{beg step2}
\end{split}
\end{align}

We introduce the notation 
$
\tilde{W}_{l,i}=C_{i}+ \frac{t}{\sqrt{N}}\sum_{j=n+K+1}^{N-1}f\circ \cT_{j,l}.
$

Let $ 0\le l_{1}\le l_{2} \le n+K$. Then we see that
\beq \label{pushing tilde W}
\tilde{W}_{l_{1},i}=\left(C_{i}+ \frac{1}{\sqrt{N}}\sum_{j=n+K+1}^{N-1}f\circ \cT_{j,l_{2}+1} \right)\circ \cT_{l_{2},l_{1}}=\tilde{W}_{l_{2}+1,i}\circ \cT_{l_{2},l_{1}}.
\eeq
Now by using the properties of the transfer operator and \eqref{pushing tilde W}
\begin{align}\begin{split}
 &\sum_{\alpha} \sum_{i}\int_{I_{i}}\left(\left( f_{\alpha}\circ \cT_{n,B+1}-\mu(f^{n}_{\alpha})\right) \partial_{\alpha} h\left(C_{i}+ \frac{t}{\sqrt{N}}\sum_{j=n+K+1}^{N-1}f\circ \cT_{j,B+1}\right)\right)\circ \cT_{B} \varrho dm 
\\
= &\sum_{\alpha} \sum_{i}\int_{\bS^{1}}\left(\left( f_{\alpha}\circ \cT_{n,B+1}-\mu(f^{n}_{\alpha})\right) \partial_{\alpha} h\left(C_{i}+ \frac{t}{\sqrt{N}}\sum_{j=n+K+1}^{N-1}f\circ \cT_{j,B+1}\right)\right)\cL_{B}(\varrho 1_{I_{i}}) dm 
\\
= &\sum_{\alpha} \sum_{i}\mu(I_{i})\int_{\bS^{1}}\left(\left(f_{\alpha}\circ \cT_{n,B+1}-\mu(f^{n}_{\alpha})\right) \partial_{\alpha} h(\tilde{W}_{B+1,i})\right) \cL_{B} \left(\frac{\varrho 1_{I_{i}}}{\mu_{0}(I_{i})}\right) dm
\\
= &\sum_{\alpha} \sum_{i}\mu(I_{i})\int_{\bS^{1}}\left(\left(f_{\alpha}-\mu(f^{n}_{\alpha})\right) \partial_{\alpha}h(\tilde{W}_{n+1,i})\right)\circ \cT_{n,B+1} \cL_{B}\left(\frac{\varrho 1_{I_{i}}}{\mu_{0}(I_{i})}\right) dm
\\
= &\sum_{\alpha} \sum_{i}\mu(I_{i})\int_{\bS^{1}}\left(\left(f_{\alpha}-\mu(f^{n}_{\alpha})\right) \partial_{\alpha}h(\tilde{W}_{n+1,i})\right)\cL_{n,B+1}\left(\cL_{B}\left(\frac{\varrho 1_{I_{i}}}{\mu_{0}(I_{i})}\right)\right) dm. \end{split}\label{pushforwarding}
\end{align}
Since $\varrho \in \cD_{L_{0}}$, 
Corollary \ref{pushforward cor} yields
\beq
\left\|\cL_{n,B+1}\left(\cL_{B}\left(\frac{\varrho 1_{I_{i}}}{\mu_{0}(I_{i})}\right)\right)-\cL_{n}(\varrho)\right\|_{L_{1}} \le B_{0}\vartheta^{n-B},
\label{eq: big B}
\eeq 
where $B_{0}=B_{0}(\lambda,A_{*},L_{0})$, $\vartheta=\vartheta(A_{*},\lambda)$. If $\lfloor n-K/2\rfloor\le 0$, then $B=0$ and
\beq
\cL_{n,B+1}\left(\cL_{B}\left(\frac{\varrho 1_{I_{i}}}{\mu_{0}(I_{i})}\right)\right)
=\cL_{n}(\varrho). \label{eq:B=0}
\eeq

From \eqref{eq: big B} and \eqref{eq:B=0} it follows that 

\beq
\left\|\cL_{n,B+1}\left(\cL_{B}\left(\frac{\varrho 1_{I_{1}}}{\mu_{0}(I_{i})}\right)\right)-\cL_{n}(\varrho)\right\|_{L_{1}} \le B_{0}\vartheta^{\left\lceil \frac{K}{2}\right\rceil}.
\label{eq: any B}
\eeq

Now $\|\left(f_{\alpha}-\mu(f^{n}_{\alpha})\right) \partial_{\alpha}h(\tilde{W}_{n+1,i})\|_{\infty}\le 2\|f\|_{\infty}\|Dh\|_{\infty}$ for every $\alpha$, and thus \eqref{beg step2}, \eqref{pushforwarding} and \eqref{eq: any B} give

\begin{align} \begin{split}
&\bigg|\sum_{\alpha} \sum_{i}\int_{I_{i}}\left(\left(f^{n}_{\alpha}-\mu(f^{n}_{\alpha})\right) \partial_{\alpha} h\left(C_{i}+ \frac{t}{\sqrt{N}}\sum_{j=n+K+1}^{N-1}f^{j}\right)\right) \varrho dm 
\\
&\,- \sum_{\alpha} \sum_{i}\mu(I_{i})\int_{\bS^{1}}\left(\left(f_{\alpha}-\mu(f_{\alpha}^{n})\right) \partial_{\alpha}h(\tilde{W}_{n+1,i})\right)\cL_{n}(\varrho) dm \bigg|
\\
\le \,&2d \|f\|_{\infty}\|Dh\|_{\infty}B_{0} \vartheta^{\left\lceil \frac{K}{2}\right\rceil}.\end{split}\label{STEP 2}
\end{align}
\noindent{\bf Step 3.} In Step 3 we use Lemma \ref{splitting lemma} to show that 
$f_{\alpha}-\mu(f_{\alpha}^{n})$ and  $\partial_{\alpha}h(\tilde{W}_{n+1,i})$ are nearly uncorrelated. Furthermore $\int_{\bS^{1}}(f_{\alpha}-\mu(f_{\alpha}^{n})) \cL_{n}(\varrho) dm=0$. These facts then yield that
\beqn
\sum_{\alpha} \sum_{i}\mu(I_{i})\int_{\bS^{1}}\left(\left(f_{\alpha}-\mu(f_{\alpha}^{n})\right) \partial_{\alpha}h(\tilde{W}_{n+1,i})\right)\cL_{n}(\varrho) dm \approx 0.
\eeqn

More precisely, first \eqref{pushing tilde W} gives
\begin{align*}
&\sum_{\alpha} \sum_{i}\mu(I_{i})\int_{\bS^{1}}\left(f_{\alpha}-\mu(f_{\alpha}^{n})\right) \partial_{\alpha}h(\tilde{W}_{n+1,i})\cL_{n}(\varrho) dm 
\\
=&\sum_{\alpha} \sum_{i}\mu(I_{i})\int_{\bS^{1}}\left(f_{\alpha}-\mu(f_{\alpha}^{n})\right)\partial_{\alpha}h(\tilde{W}_{n+K+1,i})\circ \cT_{n+K,n+1}\cL_{n}(\varrho) dm. 
\end{align*}

Lemma \ref{splitting lemma} then yields

\begin{align*}
&\bigg|\int_{\bS^{1}}\left(f_{\alpha}-\mu(f_{\alpha}^{n})\right)\partial_{\alpha}h(\tilde{W}_{n+K+1,i})\circ \cT_{n+K,n+1}\cL_{n}(\varrho) dm
\\
&\,-\int_{\bS^{1}}\left(f_{\alpha}-\mu(f^{n}_{\alpha})\right)\cL_{n}(\varrho) dm \int_{\bS^{1}} \partial_{\alpha}h(\tilde{W}_{n+K+1,i})\circ \cT_{n+K,n+1}\cL_{n}(\varrho) dm
\bigg|
\\
\le
\,&\Vert Dh \Vert_{\infty}\left(2\operatorname{Lip}(f)\lambda^{-\left\lfloor \frac{K}{2}\right\rfloor} + 2\Vert f\Vert_{\infty}B_{0}\vartheta^{\left\lceil \frac{K}{2}\right\rceil}\right).
\end{align*}

Hence $\int_{\bS^{1}}\left(f_{\alpha}-\mu(f^{n}_{\alpha})\right)\cL_{n}(\varrho) dm=0$ and $\sum_{i}\mu(I_{i})=1$, and we deduce

\begin{align}
\begin{split}
&\bigg|\sum_{\alpha} \sum_{i}\mu(I_{i})\int_{\bS^{1}}\left(f_{\alpha}-\mu(f_{\alpha}^{n})\right) \partial_{\alpha}h(\tilde{W}_{n+1,i})\cL_{n}(\varrho) dm \bigg|
\\
\le\, &d\Vert Dh \Vert_{\infty}\left(2\operatorname{Lip}(f)\lambda^{-\left\lfloor  \frac{K}{2}\right\rfloor} + 2\Vert f\Vert_{\infty}B_{0}\vartheta^{\left\lceil \frac{K}{2}\right\rceil}\right).\end{split}\label{STEP 3}
\end{align}
\noindent{\bf Step 4.}
Using the triangle inequality and the estimates collected in \eqref{STEP 1},\eqref{STEP 2} and \eqref{STEP 3}  we get
\begin{align*}
\left|\mu( \bar{f}^n \cdot  \nabla h(v + W_{n} t) )\right|\le \,&2d^{2} \|f\|_{\infty}\|D^{2}h\|_{\infty}\dfrac{\operatorname{Lip}(f) \lambda^{- \left\lfloor \frac{K}{2}\right\rfloor}}{\sqrt{N}(\lambda -1)}
\\
&+  
2d \|f\|_{\infty}\|Dh\|_{\infty}B_{0} \vartheta^{\left\lceil \frac{K}{2}\right\rceil}
\\
&+ 
d\Vert Dh \Vert_{\infty}\left(2\operatorname{Lip}(f)\lambda^{-\left\lfloor\frac{K}{2}\right\rfloor} + 2\Vert f\Vert_{\infty}B_{0}\vartheta^{\left\lceil \frac{K}{2}\right\rceil}\right)
\\
\le\, &2d^{2} \|f\|_{\infty}\|D^{2}h\|_{\infty}\dfrac{\operatorname{Lip}(f) \lambda^{-\frac{K-1}{2}}}{\sqrt{N}(\lambda -1)}
\\
&+
4d B_{0}\|f\|_{\infty}\|Dh\|_{\infty} \vartheta^{ \frac{K}{2}}
+
2d\Vert Dh \Vert_{\infty}\operatorname{Lip}(f)\lambda^{-\frac{K-1}{2}} 
\end{align*}
This finishes the proof of Theorem \ref{f nabla W lemma}.  
\end{proof}
Theorem \ref{fAW lemma} is proved with exactly same steps, by replacing $\nabla h$ by $A$, $v$ by $0$ and $t$ by ~$1$.

\subsection{Finishing the proofs of Theorems in Section \ref{circle maps}}
After calculating upper bounds for $\rho$ and $\tilde{\rho}$ 
we are now ready to prove the theorems and corollaries in Section \ref{circle maps}.

We use Theorems \ref{thm:main} and \ref{main thm 1d} to prove the results in Section \ref{circle maps}. Using those results requires choosing values of $N$ and $K$ such that $N>K$. It turns out that to minimize the upper bounds in results of Section \ref{circle maps} we need to choose $K=C\log N$, where $C$ is some constant. For small values of $N$ it might be that $C\log N>N$, therefore we have formulated the results in such way that they hold only for large enough $N$.

We are going to choose $K=\left\lceil \dfrac{2\log N}{-\log \vartheta}\right \rceil$ and the purpose of next lemma is to guarantee that this choice works in the proof as meant.
\begin{lem}\label{Choosing K Lemma}
Let $\vartheta \in \,]0,1[$. Then
\\
i) If $x\ge 3$, then $\left\lceil \dfrac{2\log x}{-\log \vartheta}\right \rceil +1\le  \dfrac{4\log x}{1-\vartheta}$.
\\
ii) If $x\ge 16/(1-\vartheta)^{2}$, then 
$
 x> \left\lceil \dfrac{2\log x}{-\log \vartheta}\right \rceil.
$
\end{lem}
\begin{proof}
i). First we introduce a following fact: If $a\ge 1$ and $b> 0$, then
\beq
\left\lceil \frac{2a}{b}\right\rceil \le \frac{3a}{\min\{1,b\}}. \label{eq:ceil inequality}
\eeq
This can be seen by studying the cases $b\le 1$ and $b> 1$ separately.
Thus it holds that 
\beqn
\left\lceil \frac{2\log x}{-\log \vartheta}\right \rceil +1 \le \frac{3\log x}{\min\{1,-\log\vartheta\}}+1\le  \frac{4\log x}{\min\{1,-\log\vartheta\}} \le \frac{4\log x}{1-\vartheta},
\eeqn
which completes the proof of i).

ii). 
Let $x_{0}= 16/(1-\vartheta)^{2}$. Then by i) 
\beqn
\left\lceil \dfrac{2\log x_{0}}{-\log \vartheta}\right \rceil \le  \dfrac{4\log x_0}{1-\vartheta}-1.
\eeqn
Since for all $y>0$ it holds that $\log y^{2}=2\log y <y$, we have
\beqn
\dfrac{4\log x_{0}}{1-\vartheta}-1= \dfrac{4\log \left(\left(\frac{4}{(1-\vartheta)}\right)^{2}\right)}{1-\vartheta}-1 \le \left(\dfrac{4}{1-\vartheta}\right)^{2}-1<x_{0}.
\eeqn
Let $x\ge x_{0}$. The derivative of $ 4\log x/(1-\vartheta)$ with respect to $x$ is $4/x(1-\vartheta)$, which is at most $(1-\vartheta)/4<1$, when $x\ge 16/(1-\vartheta)^{2}$. Thus
\beqn
\left\lceil \frac{2\log x}{-\log \vartheta}\right \rceil \le
 \dfrac{4\log x}{1-\vartheta}-1 < x_{0}+ \int_{x_{0}}^{x}\dfrac{4}{t(1-\vartheta)}dt \le  x_{0}+ \int_{x_{0}}^{x}\dfrac{1-\vartheta}{4} dt \le x_{0}+(x-x_{0})=x,
\eeqn
for every $x\ge 16/(1-\vartheta)^{2}$.
\end{proof}

\subsubsection{Proof of Theorem \ref{multivariate circle theorem}}
\begin{proof}
The proof of Theorem \ref{multivariate circle theorem} is based on applying Theorem \ref{thm:main} to the model introduced in Section \ref{circle maps}. First we verify that the assumptions of Theorem \ref{thm:main} are satisfied. 

Clearly the transformations $T_{i}$, and the functions $f$ and $h$ in Theorem \ref{multivariate circle theorem} are such that the corresponding assumptions in Theorem \ref{thm:main} hold. Assumption (A3) is also explicitly stated in Theorem \ref{multivariate circle theorem}.

Let then
\beqn
 N\ge \frac{16}{(1-\vartheta)^{2}}  \qquad \text{and}\qquad K=\left\lceil \dfrac{2\log N}{-\log \vartheta}\right \rceil
\eeqn
be fixed. By Lemma \ref{Choosing K Lemma}.ii), we have $K<N$.  We choose the functions $\rho(K)$ and $\tilde{\rho}(K)$ to be as in Lemma \ref{rho lem multivariate} and \eqref{multivar rho tilde}, respectively. As was proven in the previous section, with those choices, the Assumptions (A1) and (A2) hold. 

It is also crucial to notice that the constants $C_{2}, C_{4}$ and $B_{0}$ in those definitions do not depend on $N$ or $K$. Therefore in the forthcoming computation, every dependence on $N$ and $K$ is explicit. 

We have thus checked that the Theorem \ref{thm:main} is applicable under the setting described in Theorem \ref{multivariate circle theorem} with the choices described above. It yields
\begin{align}
&|\mu(h(W)) - \Phi_{\Sigma_{N}}(h)|\nonumber
\\
\le
&\,6d^3\max\{C_2,\sqrt{C_4}\}\left(\Vert f \Vert_{\infty} \Vert D^3 h \Vert_{\infty}+ \Vert D^2 h\Vert_{\infty}\right)\sqrt{\sum_{i=0}^{\infty} (i+1)\rho(i)}\left(\frac{K+1}{\sqrt{N}} + \sum_{i=K+1}^{\infty}\rho(i) \right)\nonumber
\\
&+ \left(2d^{2} \|D^{2}h\|_{\infty}\dfrac{\| f\|_{\operatorname{Lip}}^{2}\vartheta^{ \frac{K-1}{2}}}{\vartheta^{-1} -1}
+
4d B_{0} \| f\|_{\operatorname{Lip}}\|Dh\|_{\infty} \vartheta^{\frac{K}{2}}
+
2d\Vert Dh \Vert_{\infty}\| f\|_{\operatorname{Lip}}\vartheta^{\frac{K-1}{2}}\right) 
\sqrt{N}.
\label{eq:too long eq1}
\end{align}
Since $\rho(i)=\vartheta^{i}$ we have 
\beq
\sum_{i=K+1}^{\infty}\rho(i)= \sum_{i=K+1}^{\infty}\vartheta^{i}=\dfrac{\vartheta^{K+1}}{1-\vartheta}\label{eq:rho sum1} 
\eeq
and, by some calculations omitted here,
\beq
\sqrt{\sum_{i=0}^{\infty} (i+1)\rho(i)}\le \frac{1}{1-\vartheta}. \label{eq:rho sum2}
\eeq
Thus by \eqref{eq:rho sum1} and \eqref{eq:rho sum2}: 
\begin{align}
\begin{split}\eqref{eq:too long eq1}  \le &\,6d^3\max\{C_2,\sqrt{C_4}\}\left(\Vert f \Vert_{\infty} \Vert D^3 h \Vert_{\infty}+ \Vert D^2 h\Vert_{\infty}\right)\frac{1}{1-\vartheta}\left(\frac{K+1}{\sqrt{N}} + \dfrac{\vartheta^{K+1}}{1-\vartheta} \right)
\\
&+ \left(2d^{2} \|D^{2}h\|_{\infty}\dfrac{\| f\|_{\operatorname{Lip}}^{2}\vartheta^{ \frac{K-1}{2}}}{\vartheta^{-1} -1}
+
4d B_{0} \| f\|_{\operatorname{Lip}}\|Dh\|_{\infty} \vartheta^{\frac{K}{2}}
+
2d\Vert Dh \Vert_{\infty}\| f\|_{\operatorname{Lip}}\vartheta^{\frac{K-1}{2}}\right) 
\sqrt{N}.
\label{eq:too long eq 2}
\end{split}
\end{align}

We now make the substitution $K=\left\lceil \dfrac{2\log N}{-\log \vartheta}\right \rceil$ to \eqref{eq:too long eq 2}. Then $\vartheta^{K}\le \vartheta^{\frac{2\log N}{-\log \vartheta}} = N^{-2}$, and by Lemma \ref{Choosing K Lemma}.i)
\beqn
K+1\le \frac{4\log N}{1-\vartheta}.
\eeqn
Thus
\begin{align*}
\eqref{eq:too long eq 2}\le \, &6d^3\max\{C_2,\sqrt{C_4}\}\left(\Vert f \Vert_{\infty} \Vert D^3 h \Vert_{\infty}+ \Vert D^2 h\Vert_{\infty}\right)\frac{1}{1-\vartheta}\left(\frac{4\log N}{(1-\vartheta)\sqrt{N}} + \dfrac{N^{-2}}{(1-\vartheta)} \right)
\\
&+ \left(2d^{2} \|D^{2}h\|_{\infty}\dfrac{\| f\|_{\operatorname{Lip}}^{2}N^{-1}}{\vartheta^{-\frac{1}{2}} -\vartheta^{\frac{1}{2}}}
+
4d B_{0} \| f\|_{\operatorname{Lip}}\|Dh\|_{\infty} N^{-1}
+
\frac{2d\Vert Dh \Vert_{\infty}\| f\|_{\operatorname{Lip}}N^{-1}}{\vartheta^{\frac{1}{2}}}\right)\sqrt{N}.
\\
\begin{split}
\le &\,\Bigg(\frac{6d^3\max\{C_2,\sqrt{C_4}\}\left(\Vert f \Vert_{\infty} \Vert D^3 h \Vert_{\infty}+ \Vert D^2 h\Vert_{\infty}\right)}{{1-\vartheta}}\left(\frac{4\log N}{1-\vartheta} + \dfrac{1}{1-\vartheta}\right) 
\\
&+ \left(2d^{2} \|D^{2}h\|_{\infty}\dfrac{\| f\|_{\operatorname{Lip}}^{2}}{\vartheta^{-\frac{1}{2}} -\vartheta^{\frac{1}{2}}}
+
4d B_{0} \| f\|_{\operatorname{Lip}}\|Dh\|_{\infty}
+
\frac{2d\Vert Dh \Vert_{\infty}\| f\|_{\operatorname{Lip}}}{\vartheta^{\frac{1}{2}}}\right)
\Bigg)N^{-\frac{1}{2}}.
\end{split}
\end{align*}
Since we assumed $N\ge \dfrac{16}{(1-\vartheta)^{2}} \ge e$, we have $\log N\ge 1$ which finally yields
\begin{align}
&|\mu(h(W)) - \Phi_{\Sigma_{N}}(h)|\nonumber
\\
\begin{split}
\le &\,\Bigg(\frac{30d^3\max\{C_2,\sqrt{C_4}\}\left(\Vert f \Vert_{\infty} \Vert D^3 h \Vert_{\infty}+ \Vert D^2 h\Vert_{\infty}\right)}{(1-\vartheta)^{2}}
\\
&+ 2d^{2} \|D^{2}h\|_{\infty}\dfrac{\| f\|_{\operatorname{Lip}}^{2}}{\vartheta^{-\frac{1}{2}} -\vartheta^{\frac{1}{2}}}
+
4d B_{0} \| f\|_{\operatorname{Lip}}\|Dh\|_{\infty} \label{multivar circle constant}
+
\frac{2d\Vert Dh \Vert_{\infty}\| f\|_{\operatorname{Lip}}}{\vartheta^{\frac{1}{2}}}
\Bigg)N^{-\frac{1}{2}}\log N.
\end{split}
\end{align}
Since \eqref{multivar circle constant} holds for all $N\ge \dfrac{16}{(1-\vartheta)^{2}} $, we have now completed the proof of Theorem \ref{multivariate circle theorem}.
\end{proof}

\subsubsection{Proofs of Theorem \ref{circle main thm} and Corollary \ref{N^1/6 convergence cor}}

The proof of Theorem \ref{circle main thm} proceeds similarly to the proof of Theorem \ref{multivariate circle theorem}. As in the previous proof let 
\beqn
 N\ge \frac{16}{(1-\vartheta)^{2}} \qquad \text{and}\qquad K=\left\lceil \dfrac{2\log N}{-\log \vartheta}\right \rceil
\eeqn
be fixed,  and let $p\ge 0$ and $C_{0}>0$ be such that $\sigma_{N}\ge C_{0}N^{-p}$.  Define functions $\rho(K)$ and $\tilde{\rho}(K)$ as in Lemma \ref{rho lem multivariate} and \eqref{univar rho tilde}, respectively. The assumptions of Theorem \ref{main thm 1d} are now satisfied as the reader may verify. By reusing the results in the proof of Theorem \ref{multivariate circle theorem}, it yields

\begin{align*}
&d_\mathscr{W}(W,\sigma_{N}Z)
\\
\le\,&
 12 \max\{\sigma_{N}^{-1},\sigma_{N}^{-2}\} \max\{C_2,\sqrt{C_4}\} (1 + \Vert f \Vert_{\infty}) \sqrt{\sum_{i = 0}^{\infty}(i+1) \rho(i)}\!\left(\frac{K+1}{\sqrt{N}} + \sum_{i= K+1}^{\infty}  \rho(i) \right)
\\
&+ 2\max\{1,\sigma_{N}^{-2}\}\sqrt N\tilde\rho(K)
\\
\le\,&
 \frac{12 \max\{\sigma_{N}^{-1},\sigma_{N}^{-2}\} \max\{C_2,\sqrt{C_4}\} (1 + \Vert f \Vert_{\infty}) }{1-\vartheta}\left(\frac{4\log N}{(1-\vartheta)\sqrt{N}} + \dfrac{N^{-2}}{1-\vartheta} \right)
\\
&+ 2\max\{1,\sigma_{N}^{-2}\}\sqrt N \left(2\dfrac{\| f\|_{\operatorname{Lip}}^{2}\vartheta^{ \frac{K-1}{2}}}{\vartheta^{-1} -1}
+
 4B_{0} \| f\|_{\operatorname{Lip}} \vartheta^{\frac{K}{2}}
+
2\| f\|_{\operatorname{Lip}}\vartheta^{\frac{K-1}{2}}\right)
\\
\le\,&
  \frac{60\max\{C_2,\sqrt{C_4}\} (1 + \Vert f \Vert_{\infty}) }{(1-\vartheta)^{2}} \max\{\sigma_{N}^{-1},\sigma_{N}^{-2}\}N^{-\frac{1}{2}} \log N 
\\
&+ \left(4\dfrac{\| f\|_{\operatorname{Lip}}^{2}}{\vartheta^{-\frac{1}{2}} -\vartheta^{\frac{1}{2}}}
+
8 B_{0} \| f\|_{\operatorname{Lip}}
+
\frac{4\| f\|_{\operatorname{Lip}}}{\vartheta^{\frac{1}{2}}}\right) \max\{1,\sigma_{N}^{-2}\}N^{-\frac{1}{2}}.
\end{align*} 
We have $\max\{\sigma_{N}^{-1},\sigma_{N}^{-2}\}\le\max\{1,\sigma_{N}^{-2}\}$ and thus
\beq
d_\mathscr{W}(W,\sigma_{N}Z)
\le
 \tilde{C}\max\{1,\sigma_{N}^{-2}\}N^{-\frac{1}{2}}\log N,\label{univ formula}
\eeq
where
\beqn
\tilde{C}=\frac{60\max\{C_2,\sqrt{C_4}\} (1 + \Vert f \Vert_{\infty}) }{(1-\vartheta)^{2}}+\dfrac{4\| f\|_{\operatorname{Lip}}^{2}}{\vartheta^{-\frac{1}{2}} -\vartheta^{\frac{1}{2}}}
+
8 B_{0} \| f\|_{\operatorname{Lip}}
+
\frac{4\| f\|_{\operatorname{Lip}}}{\vartheta^{\frac{1}{2}}}.
\eeqn

Theorem \ref{circle main thm} now follows: Since it was assumed that $\sigma_{N}\ge C_{0}N^{-p}$ it holds that $\max\{1,\sigma_{N}^{-2}\}\le \max\{1,C_{0}^{-2}\}N^{2p} $ and by \eqref{univ formula} we have 
\beqn
d_\mathscr{W}(W,\sigma_{N}Z)\le \tilde{C}\max\{1,C_{0}^{-2}\}N^{-\frac{1}{2}+2p}\log N. \eeqn
Furthermore, notice that $\tilde{C}$ does not depend on $N$.

The idea behind the proof of Corollary \ref{N^1/6 convergence cor} is as follows:

 If the variance of $W$ is large, then Theorem \ref{circle main thm} gives good upper bound to Wasserstein distance $d_{\mathscr{W}}(W,\sigma_{N}Z)$. On the other hand if $\sigma_{N}$ is close to zero, then both the distribution of $W$ and $\sigma_{N}Z$ are close to the Dirac delta distribution $\delta_{0}$ in the sense of Wasserstein distance. This gives us two distinct ways to find upper bound to $d_{\mathscr{W}}(W,\sigma_{N}Z)$. It turns out that the worst-case scenario happens when variance behaves like $CN^{-\frac{1}{6}}$. 

To handle the small values of $\sigma_{N}$, we introduce the following fact:

Let $X$ and $Y$ be two random variables with means $0$ and variances $\sigma^{2}_{X},\sigma^{2}_{Y}$, respectively. Then 
\beq
d_{\mathscr{W}}(X,Y)\le \sigma_{X}+\sigma_{Y}.\label{upper bound by variance}
\eeq
To see this, assume that $X_{0}$ is a random variable such that $P(X_{0}=0)=1$. Then
\begin{align*}
d_{\mathscr{W}}(X_{0},X)
&= \sup_{h\in \mathscr{W}}\left| \int h(x)dF_{X_{0}}(x)- \int h(x)dF_{X}(x)\right|
\\
&=\sup_{h\in \mathscr{W}}\left| h(0)-\bE [h(X)]\right|
\leq \bE[\left| X\right|] 
\leq \sqrt{\bE[ X^{2}]}
=\sigma_{X}.
\end{align*}
Now \eqref{upper bound by variance} follows from triangle inequality, since $d_{\mathscr{W}}$ is a metric. 

Assume that $\sigma_{N}\ge CN^{-p}$, where $p$ and $C$ are some constants. Then the Wasserstein distance $d_{\mathscr{W}}(W,Z)$ has an upper bound of type $CN^{-\frac{1}{2}+2p}\log N$ by Theorem \ref{circle main thm}. On the contrary, if $\sigma_{N}< CN^{-p}$ then  formula \eqref{upper bound by variance} gives an upper bound of type $CN^{-p}$. Since the equation $-\frac{1}{2}+2p= -p$ is solved by $p=\frac{1}{6}$, we make the following choices: 
 
Let $p=\frac{1}{6}$.
Choose $C_{0}=1$ in Theorem \ref{circle main thm}. Then
\begin{align}
\begin{split}
d_{\mathscr{W}}(W,\sigma_{N}Z)
\le 
\tilde{C}N^{-\frac{1}{6}}\log N, \label{N^1/6 first bound}
\end{split}
\end{align}
when $\sigma_{N}\ge N^{-\frac{1}{6}}$. If $\sigma_{N}< N^{-\frac{1}{6}}$, then by \eqref{upper bound by variance}
\begin{align}
d_{\mathscr{W}}(W,\sigma_{N}Z)
\le 2N^{-\frac{1}{6}}\le 2N^{-\frac{1}{6}}\log N.\label{N^1/6 second bound}
\end{align}
Since either \eqref{N^1/6 first bound} or \eqref{N^1/6 second bound} holds, we have
\beq
d_{\mathscr{W}}(W,\sigma_{N}Z)\le    \max\left\{\tilde{C}, 2 \right\}N^{-\frac{1}{6}}\log N.\label{N^1/6 constant}
\eeq
This proves Corollary \ref{N^1/6 convergence cor}.
\subsubsection{Proof of Corollary \ref{Self-normalized cor}}

Let $C_{0}'>0$, $r\ge 0$ and $s_{N}^{2}= N\sigma^{2}_{N}>C_{0}'N^{r}$, which implies that $\sigma_{N}=\dfrac{s_{N}}{\sqrt{N}}>\sqrt{C_{0}'}N^{\frac{r-1}{2}}$. Then using  properties of Wasserstein distance gives
\beqn
 d_{\mathscr{W}}\left(\frac{S_{N}}{s_{N}},Z\right) =  d_{\mathscr{W}}\left(\frac{W}{\sigma_{N}},Z\right) 
= \sigma_{N}^{-1} d_{\mathscr{W}}\left(W,\sigma_{N}Z\right) \le C_{0}'^{-\frac12}N^{\frac{1-r}{2}} d_{\mathscr{W}}\left(W,\sigma_{N}Z\right).
\eeqn
 We may now apply Theorem \ref{circle main thm} to $d_{\mathscr{W}}\left(W,\sigma_{N}Z\right)$, with values $p=(1-r)/2$ and $C_{0}=\sqrt{C_{0}'}$ which yields
\begin{align*}
 d_{\mathscr{W}}\left(\frac{S_{N}}{s_{N}},Z\right) &\le C_{0}'^{-\frac12}N^{\frac{1-r}{2}} d_{\mathscr{W}}\left(W,\sigma_{N}Z\right)
\\
&\le C_{0}'^{-\frac12}N^{\frac{1-r}{2}}
(\tilde{C}\max\{1,C_{0}'^{-1}\}N^{-\frac{1}{2}+1-r} \log N)
\\
&\le \tilde{C}\max\{C_{0}'^{-\frac{1}{2}},C_{0}'^{-\frac{3}{2}}\}N^{1-\frac{3r}{2}} \log N.
\end{align*}

This completes the proof of Corollary \ref{Self-normalized cor}.

\section{Proofs for Application II}\label{Proofs of quasistatic model results}
In this section we use the notation defined in Section \ref{quasistatic model}. The reader should recall the definitions of $\cT_{n,i}, \cT_{n,i,j}$ and $f_{n,i}$ from that section to avoid confusion with the notations used on Sections \ref{circle maps} and \ref{computing rho}. The pushforward measure $(\cT_{n,k})_{*}\mu$ is denoted by $\mu_{n,k}$ and the corresponding density by $\varrho_{n,k}$.

The density $\hat{\varrho}_{t}$ of the SRB measure $\hat{\mu}_{t}$ is Lipschitz continuous by Remark 4.1 of \cite{DobbsStenlund_2016}. By the same remark $\cL_{t}^{k}1$ converges to $\hat{\varrho}_{t}$ in the supremum norm and thus $\hat{\varrho}_{t}>0$. Furthermore by Remark 4.4.(iii) of \cite{DobbsStenlund_2016} $\hat{\varrho}_{t}\in \cD_{L}$ for some $L\ge 0$. Since $\cL_{t}^{k}\hat{\varrho}_{t}= \hat{\varrho}_{t}$ for all $k\ge 0$, by \eqref{L* property}, we have $\hat{\varrho}_{t}\in \cD_{L^{*}}$. 
\subsection{Preliminary results}\label{Hölder continuity of variance}

By duality \eqref{transfer op rule}:
\beq
\hat{\sigma}^{2}_{s}(f)=\hat{\mu}_{s}[\hat{f}_{s}^{2}]+ 2\sum_{k=1}^{\infty}m[\hat{f}_{s}\cL_{s}^{k}(\hat{\varrho}_{s}\hat{f}_{s})]=
\hat{\mu}_{s}[f^{2}] - \hat{\mu}_{s}[f]^{2} + 2\sum_{k=1}^{\infty}\left\{\hat{\mu}_{s}[f f\circ \gamma_{s}^{k}] - \hat{\mu}_{s}[f]^{2}\right\}. \label{variance as a sum} \eeq
For later use, notice that $\hat{\sigma}^{2}_{s}(f)$ can be represented in the integral form
\beq
n\int_{-\infty}^{\infty}\hat{\mu}_{s}(f f\circ\gamma_{s}^{|\lfloor ns\rfloor -\lfloor n(s+r) \rfloor|})-\hat{\mu}_{s}(f)^{2} dr,\label{integral variance}
\eeq
where $n=1,2,...$

In Lemma 6.1 of \cite{DobbsStenlund_2016} it is proven that $t\rightarrow \hat{\sigma}_{t}^{2}(f)$ is uniformly continuous. We improve the proof to show that it is even Hölder continuous.
\begin{lem}
$t\rightarrow \hat{\sigma}_{t}^{2}(f)$ is Hölder continuous with every exponent $\eta''<\eta$. An upper bound for the corresponding Hölder constant $C$ can be given as a function of $\lambda, A_{*},\eta,\eta'', C_{H}$ and $\|f\|_{\operatorname{Lip}}$. \label{Hölder continuity lemma}
\end{lem}

\begin{proof}
Let $k\ge 0$. We have 
$
m[\hat{f}_{t}\cL_{t}^{k}(\hat{\varrho}_{t}\hat{f}_{t})] =m[f\cL_{t}^{k}(\hat{\varrho}_{t}f)]-m[\hat{\varrho}_{t}f]^{2}.
$
The computation given in the proof of Lemma 6.1 in \cite{DobbsStenlund_2016} yields
\beqn
\left| m(f\cL_{t}^{k}(\hat{\varrho}_{t}f))-m(f\cL_{s}^{k}(\hat{\varrho}_{s}f)) \right|
\le
\|f\|_{\operatorname{Lip}}^{2}\left(\|\cL_{t}^{k}-\cL_{s}^{k}\|_{\operatorname{Lip}\rightarrow C^{0}}\|\hat{\varrho}_{t}\|_{\operatorname{Lip}}+\|\hat{\varrho}_{t}- \hat{\varrho}_{s}\|_{L^{1}(m)} \right).
\eeqn
Let $\eta'<\eta$ and $k\ge 1$. By (19), (8) and (22) of \cite{DobbsStenlund_2016} the right side can be approximated from above by
\beqn
C(kd_{C^{1}}(\gamma_{t},\gamma_{s})+|t-s|^{\eta'})\le C(k|t-s|^{\eta'}+|t-s|^{\eta'})\le Ck|t-s|^{\eta'}, \label{eq:k>0}
\eeqn
where $C=C(\lambda, A_{*},C_{H},\|f\|_{\operatorname{Lip}},\eta')$. Using the same results for $k=0$ it also follows that 
\begin{align*}
\left|\hat{\mu}_{t}(\hat{f}^{2}_{t})-\hat{\mu}_{s}(\hat{f}^{2}_{s})\right|
=\left|m(\hat{\varrho}_{t}f^{2})-m(\hat{\varrho}_{t}f)^{2}+m(\hat{\varrho}_{s}f)^{2}-m(\hat{\varrho}_{s}f^{2})\right|
\le C|t-s|^{\eta'}.\label{eq:k=0}
\end{align*}
Furthermore we have that for every $M\in \bN$
\beqn
\sum_{k=M}^{\infty}m(\hat{f}_{t}\cL_{t}^{k}(\hat{\varrho}_{t}\hat{f}_{t}))=\sum_{k=M}^{\infty}m(f\cL_{t}^{k}(\hat{\varrho}_{t}f))-m(\hat{\varrho}_{t}f)^{2}\le C\sqrt{\vartheta}^{M}\label{eq:k sum}
\eeqn
by Lemma \ref{splitting lemma}.
Combining all these observations, formula \eqref{variance as a sum} yields
\beqn
\left|\hat{\sigma}_{t}^{2}(f)-\hat{\sigma}_{s}^{2}(f)\right|\le\sum_{k=0}^{M-1}MC|t-s|^{\eta'}+C\sqrt{\vartheta}^{M}\le C(M^2|t-s|^{\eta'}+\sqrt{\vartheta}^{M}) 
\eeqn
for all $M=1,2,...$, where $C=C(\lambda, A_{*},C_{H},\|f\|_{\operatorname{Lip}},\eta')$. Choosing $C$ large enough,
\beqn
\left|\hat{\sigma}_{t}^{2}(f)-\hat{\sigma}_{s}^{2}(f)\right|\le C(M^2|t-s|^{\eta'}+\sqrt{\vartheta}^{M})
\eeqn
holds also for all real numbers $M\geq 0$. To prove Hölder-continuity, we choose $M$ depending on $|t-s|>0$ in the following way: 
\beqn
|t-s|^{\eta'}=\sqrt{\vartheta}^{M}\Rightarrow M=\frac{\log|t-s|^{\eta'}}{\log\sqrt{\vartheta}}=\frac{\eta'}{\log\sqrt{\vartheta}}\log|t-s|>0.
\eeqn
Using the well known fact that for all $x\in\,]0,1]$ and $\alpha\in \,]0,1[$ there exists a constant $C=C(\alpha)$ such that
\beqn
x|\log x| \le C x^\alpha, 
\eeqn
we deduce that
\beqn
M^2 = \frac{\eta'^{2}}{\log^{2}\sqrt{\vartheta}}\log^{2}|t-s|\le \frac{\eta'^{2}}{\log^{2}\sqrt{\vartheta}} C|t-s|^{2\alpha-2},
\eeqn
where $C=C(\lambda, A_{*},\eta', \alpha)$.
Let $0<\eta''<\eta' $.
Choose $\alpha=1-(\eta'-\eta'')/2$. Then 
\begin{align*}
\left|\hat{\sigma}_{t}^{2}(f)-\hat{\sigma}_{s}^{2}(f)\right|
&\le C(M^2|t-s|^{\eta'}+\sqrt{\vartheta}^{M}) \le C\left(\frac{\eta'^{2}}{\log^{2}\sqrt{\vartheta}} C|t-s|^{\eta'+2\alpha-2}+|t-s|^{\eta'}\right) 
\\
&\le C|t-s|^{\eta''},
\end{align*}
where, in the rightmost expression, $C=C(\lambda, A_{*},C_{H},\|f\|_{\operatorname{Lip}},\eta',\eta'')$.
   
Since $\eta''$ can be arbitrarily close to $\eta'$ and $\eta'$ arbitrary close to $\eta$ we see that $t\mapsto \hat{\sigma}_{t}^{2}(f)$ is Hölder continuous in $[0,1]$ for any Hölder exponent $\eta''< \eta$. The result now follows by choosing for example $\eta'=(\eta+\eta'')/2$.
\end{proof}

The next lemma follows from Lemma 5.9 in \cite{DobbsStenlund_2016}
\begin{lem}\label{b lemma}
There exists a constant $b>0$ such that the following holds. Given $\eta'<\eta$ there exists $C=C(\lambda,A_{*},C_{H},\eta',L_{0})$ such that
\beqn
\| \varrho_{n,\lfloor nt\rfloor} -\hat{\varrho}_{s}\|_{L^{1}}\le C(n^{-\eta'}+|t-s|^{\eta'}),
\eeqn
when $t\ge bn^{-1}\log n$.
\end{lem}

Recall that the variance of $\xi_{n}(t)$ with respect to $\mu$ is denoted by $\sigma^{2}_{n,t}$. Since $\xi_{n}(t)$ is a sum of random variables with mean $0$, we have $\mu(\xi_{n}(t))=0$ and that
$
\sigma_{n,t}^{2}=\mu\left((\xi_{n}(t))^{2}\right).
$

Next we approximate 
\beqn
\left|\sigma_{n,t}^{2}-\sigma_{t}^{2} \right|.
\eeqn
The proof of the following lemma follows that of Lemma 6.2 of \cite{DobbsStenlund_2016}. We need a more explicit version in this paper.
\begin{lem}
Let $\eta''<\eta$. Then there exists a constant $C=C(\lambda,A_{*},C_{H},\eta,\eta'',L_{0}, \|f\|_{\operatorname{Lip}})$ such that
\beqn
\left|\sigma_{n,t}^{2}-\sigma_{t}^{2}\right| \le Cn^{-\eta''}
\eeqn
for every $t\in [0,1]$.

\label{variance distance lemma}
\end{lem}
\begin{proof}

Let $\eta'<\eta$.
We have
\beqn
\sigma_{n,t}^{2}=\mu\left((\xi_{n}(t))^{2}\right)= n\int_{0}^{t}\int_{0}^{t}\mu(\bar{f}_{n,\lfloor ns\rfloor}\bar{f}_{n,\lfloor nr\rfloor})dr ds.
\eeqn

Let $\kappa\in \,]0,\frac{1}{4}[$ satisfy $ 2\kappa < \eta'(1-\kappa)$ and define $a_{n}=n^{-1+\kappa}$. Then $a_{n}>bn^{-1} \log n$ for big enough $n$, where $b$ is the same constant as in Lemma \ref{b lemma}. Define the sets 
\beqn
P_{n}=\lbrace (s,r)\in [0,t]^{2}: 2a_{n}\le s \le t-a_{n} \text{ and } |r-s|\le a_{n} \rbrace,
\eeqn
\beqn
Q_{n}=\lbrace (s,r)\in [0,t]^{2}: |r-s|\le a_{n} \text{ and either } s< 2a_{n} \text{ or } s > t-a_{n}   \rbrace
\eeqn
 and 
\beqn
R_{n}=\lbrace (s,r)\in [0,t]^{2}: |r-s|> a_{n} \rbrace.
\eeqn
Notice that $P_{n}\cup R_{n}\cup Q_{n}=[0,t]^{2}$.
The area of $Q_{n}$ is at most $6a_{n}^{2}$ and $|f_{n,\lfloor ns\rfloor}f_{n,\lfloor nr\rfloor}|\le \|f\|^{2}_{\infty}$. Thus
\beqn
\left| n\int\int_{Q_{n}}\mu(\bar{f}_{n,\lfloor ns\rfloor}\bar{f}_{n,\lfloor nr\rfloor})dr ds \right| \le 6\|f\|_{\infty}^{2}na_{n}^{2}=6\|f\|_{\infty}^{2}n^{-1+2\kappa}.
\eeqn
From now on $E$ denotes a real valued function such that there exists a constant  $C=C(\lambda,A_{*},C_{H},\eta',L_{0}, \kappa, \|f\|_{\operatorname{Lip}}) >0$ such that $|E|\le C$. The specific formulas for values of $C$ might change from line to line in the computation.

By Lemma 5.10 in \cite{DobbsStenlund_2016} we know that
\beq
| \mu(\bar{f}_{n,\lfloor ns\rfloor}\bar{f}_{n,\lfloor nr\rfloor})| \le E\vartheta^{n|r-s|}.\label{off diagonal}
\eeq

By \eqref{off diagonal} we see that
\beqn
\left| n\int\int_{R_{n}}\mu(\bar{f}_{n,\lfloor ns\rfloor}\bar{f}_{n,\lfloor nr\rfloor})dr ds \right|\le E\vartheta^{na_{n}}= E\vartheta^{n^{\kappa}}.
\eeqn
 For large enough $n$, we have $E\vartheta^{n^{\kappa}}\le 6\|f\|_{\infty}^{2}n^{-1+2\kappa}$. Thus 
\beqn
\left| n\int\int_{Q_{n}\cup R_{n}}\mu(\bar{f}_{n,\lfloor ns\rfloor}\bar{f}_{n,\lfloor nr\rfloor})dr ds \right| \le 12\|f\|_{\infty}^{2}na_{n}^{2}=Ena_{n}^{2}.
\eeqn

The only major contribution to the integral now comes from $P_{n}$, i.e. 
\beq
n\int_{0}^{t}\int_{0}^{t}\mu(\bar{f}_{n,\lfloor ns\rfloor}\bar{f}_{n,\lfloor nr\rfloor})dr ds=n\int_{2a_{n}}^{t-a_{n}}\int_{s-a_{n}}^{s+a_{n}}\mu(\bar{f}_{n,\lfloor ns\rfloor}\bar{f}_{n,\lfloor nr\rfloor})dr ds + Ena_{n}^{2}. \label{equ1}
\eeq
Next we will show that $n\int_{s-a_{n}}^{s+a_{n}}\mu(\bar{f}_{n,\lfloor ns\rfloor}\bar{f}_{n,\lfloor nr\rfloor})dr\approx \hat{\sigma}_{s}^{2}$:

 By Lemma \ref{b lemma} we have
\beqn
\|\varrho_{n,\lfloor nr\rfloor}-\hat{\varrho}_{s}\|_{L^{1}}= E(n^{-\eta'}+|r-s|^{\eta'}),
\eeqn 
when $r>bn^{-1}\log n$. Thus 
\beqn
\sup_{r\in (s-a_{n},s+a_{n})}\|\varrho_{n,\lfloor nr\rfloor}-\hat{\varrho}_{s}\|_{L^{1}}= E(n^{-\eta'}+a_{n}^{\eta'}) = Ea_{n}^{\eta'}.
\eeqn
From this it follows that
\begin{align}
\begin{split}
&n\int_{s-a_{n}}^{s+a_{n}}\mu(\bar{f}_{n,\lfloor ns\rfloor}\bar{f}_{n,\lfloor nr\rfloor})dr 
\\
=\, &n\int_{s-a_{n}}^{s+a_{n}}\mu({f}_{n,\lfloor ns\rfloor}{f}_{n,\lfloor nr\rfloor})-\mu({f}_{n,\lfloor ns\rfloor})\mu({f}_{n,\lfloor nr\rfloor})dr
\\
=\,& n\int_{s-a_{n}}^{s+a_{n}}\mu({f}_{n,\lfloor ns\rfloor}{f}_{n,\lfloor nr\rfloor})-\hat{\mu}_{s}(f)^{2}dr+ n\int_{s-a_{n}}^{s+a_{n}} \hat{\mu}_{s}(f)^{2}-  \mu({f}_{n,\lfloor ns\rfloor})\mu({f}_{n,\lfloor nr\rfloor})dr
\\
=\,& n\int_{s-a_{n}}^{s+a_{n}}\mu({f}_{n,\lfloor ns\rfloor}{f}_{n,\lfloor nr\rfloor})-\hat{\mu}_{s}(f)^{2}dr +Ena_{n}^{1+\eta'}.\label{eq: bars of}
\end{split}
\end{align}

Define $b_{n}=\frac{1}{n}(1-\lbrace ns\rbrace )$.  We have  
\begin{align*}
&n\int_{s}^{s+a_{n}}\mu(f_{n,\lfloor ns\rfloor}f_{n,\lfloor nr\rfloor})dr
= n\int_{0}^{a_{n}}\mu(f_{n,\lfloor ns\rfloor}f_{n,\lfloor n(s+r)\rfloor})dr
\\
=\, &b_{n}n\mu_{n,\lfloor ns\rfloor}(f^{2})+ n\int_{b_{n}}^{a_{n}}\mu_{n,\lfloor ns\rfloor}(f f\circ T_{n,\lfloor n(s+r)\rfloor}\circ \cdots \circ T_{n,\lfloor ns\rfloor +1} )dr
\\
=\, &b_{n}n\hat{\mu}_{s}(f^{2})+ n\int_{b_{n}}^{a_{n}}\hat{\mu}_{s}(f f\circ T_{n,\lfloor n(s+r)\rfloor}\circ \cdots \circ T_{n,\lfloor ns\rfloor +1} )dr + E(a_{n}^{\eta'}+na_{n}^{1+\eta'})
\\
=\, &n\int_{0}^{b_{n}}m(f\hat{\varrho}_{s}f) dr + n\int_{b_{n}}^{a_{n}}m(f\cL_{n,\lfloor n(s+r)\rfloor}\cdots \cL_{n,\lfloor ns\rfloor +1}(\hat{\varrho}_{s}f)) dr + E(na_{n}^{1+\eta'}).
\end{align*}
We want to replace $\cL_{n,\lfloor n(s+r)\rfloor}\cdots \cL_{n,\lfloor ns\rfloor +1}$ by $\cL_{\gamma_s}
^{\lfloor n(s+r)\rfloor -\lfloor ns \rfloor}$. For this purpose notice that for every $j\in \{ 
\lfloor ns \rfloor,...,\lfloor n(s+r) \rfloor \}$ and $(r/n)\le a_{n}$ we have $d_{C^{1}}(\gamma_{s}, 
T_{n,j})\le d_{C^{1}}(\gamma_{s},\gamma_{j})+d_{C^{1}}(\gamma_{j}, T_{n,j})\le E(r/n)^{\eta} + En^{-
\eta}\le Ea_{n}^{\eta}$. 
We have
\begin{align*}
&\|\cL_{n,\lfloor n(s+r)\rfloor}\cdots \cL_{n,\lfloor ns\rfloor +1}(\hat{\varrho}_{s}f)- \cL_{s}
^{\lfloor n(s+r)\rfloor -\lfloor ns \rfloor}(\hat{\varrho}_{s}f)\|_{L^{1}(m)}
\le Enra_{n}^{\eta}= Ena_{n}^{\eta' +1}.
\end{align*} 
Hence,
\beq
n\int_{s}^{s+a_{n}}\mu(f_{n,\lfloor ns \rfloor}f_{n,\lfloor nr \rfloor}) dr = n\int_{0}^{a_{n}}m(f\cL_{s}^{\lfloor n(s+r)\rfloor -\lfloor ns \rfloor}(\hat{\varrho}_{s}f)) dr + En^{2}a_{n}^{2+\eta'}.\label{eq:an}
\eeq
By a similar computation
\beq
n\int_{s-a_{n}}^{s}\mu(f_{n,\lfloor ns \rfloor}f_{n,\lfloor nr \rfloor}) dr = n\int_{-a_{n}}^{0}m(f\cL_{s}^{\lfloor ns\rfloor -\lfloor n(s+r) \rfloor}(\hat{\varrho}_{s}f)) dr + En^{2}a_{n}^{2+\eta'}.\label{eq:-an}
\eeq
Thus by \eqref{eq: bars of}, \eqref{eq:an},\eqref{eq:-an} and using the formula \eqref{integral variance} for the variance, we have 
\begin{align}
\begin{split}
&n\int_{s-a_{n}}^{s+a_{n}}\mu(\bar{f}_{n,\lfloor ns \rfloor}\bar{f}_{n,\lfloor nr \rfloor}) dr
\\
=\,& n\int_{s-a_{n}}^{s+a_{n}}\mu({f}_{n,\lfloor ns\rfloor}{f}_{n,\lfloor nr\rfloor})-\hat{\mu}_{s}(f)^{2}dr +Ena_{n}^{1+\eta'}
\\
=\, &n\int_{-a_{n}}^{a_{n}}m(f\cL_{s}^{|\lfloor ns\rfloor -\lfloor n(s+r) \rfloor|}(\hat{\varrho}_{s}f))-\hat{\mu}_{s}(f)^{2} dr + Ena_{n}^{1+\eta'}+ En^{2}a_{n}^{2+\eta'}
\\
=\,&n\int_{-a_{n}}^{a_{n}}\hat{\mu}_{s}(f f\circ\gamma_{s}^{|\lfloor ns\rfloor -\lfloor n(s+r) \rfloor|})-\hat{\mu}_{s}(f)^{2} dr + En^{2}a_{n}^{2+\eta'}
\\
=\,&n\int_{-\infty}^{\infty}\hat{\mu}_{s}(f f\circ\gamma_{s}^{|\lfloor ns\rfloor -\lfloor n(s+r) \rfloor|})-\hat{\mu}_{s}(f)^{2} dr + En^{2}a_{n}^{2+\eta'}+ E\vartheta^{na_{n}}
\\
=\,&\hat{\sigma}_{s}^{2}(f) + En^{2}a_{n}^{2+\eta'}.\label{equ2}
\end{split}
\end{align}
Note that we can choose an upper bound for $|E|$ that is independent of $s$. This is because $\hat{\varrho}_{s}\in \cD_{L_{*}}$.

Therefore by \eqref{equ1} and \eqref{equ2} 
\begin{align*}
\mu\left((\xi_{n}(t))^{2}\right)&= n\int_{2a_{n}}^{t-a_{n}}\int_{s-a_{n}}^{s+a_{n}}\mu(\bar{f}_{n,\lfloor ns\rfloor}\bar{f}_{n,\lfloor nr\rfloor})dr ds + Ena_{n}^{2}
\\
&= \int_{2a_{n}}^{t-a_{n}}\hat{\sigma}_{s}^{2}(f) + En^{2}a_{n}^{2+\eta'}ds + Ena_{n}^{2}
\\
&= \int_{2a_{n}}^{t-a_{n}}\hat{\sigma}_{s}^{2}(f) ds + Ena_{n}^{2}+En^{2}a_{n}^{2+\eta'}
\\
&= \int_{0}^{t}\hat{\sigma}_{s}^{2}(f) ds + Ena_{n}^{2}+En^{2}a_{n}^{2+\eta'}+Ea_{n}
\\
&= \sigma_{t}^{2}(f)  + Ena_{n}^{2}+En^{2}a_{n}^{2+\eta'}
\\
&= \sigma_{t}^{2}(f)  + En^{-1+2\kappa}+En^{2\kappa-\eta'(1-\kappa)}.
\end{align*}
Let $0<\eta''<\eta'$. Recall that we have assumed that $\kappa\in \,]0,\frac{1}{4}[, 2\kappa < \eta'(1-\kappa)$ and $\eta'<\eta$. By choosing $\eta'=(\eta+\eta'')/2$ and $\kappa= (\eta-\eta'')/(4(1+\eta))$ these assumptions are satisfied as the reader may check, and we have 
\beqn
n^{-1+2\kappa}=n^{-1+\frac{\eta-\eta''}{2(1+\eta)}}=n^{\frac{-2-\eta-\eta''}{2(1+\eta)}}\le n^{\frac{-4\eta''}{4}}=n^{-\eta''} 
\eeqn
and
\beqn
 n^{2\kappa-\eta'(1-\kappa)}= n^{ \frac{\eta-\eta''}{2(1+\eta)} +\frac{\eta+\eta''}{2}\left(\frac{\eta-\eta''}{4(1+\eta)}-1\right)}= n^{\frac{4(\eta-\eta'')+ (\eta+\eta'')(-4-3\eta-\eta'))}{8(1+\eta)}}=n^{\frac{8\eta''-4\eta\eta''-(\eta'')^{2}-3\eta^{2}}{8(1+\eta)}}\le n^{-\eta''}  . 
\eeqn
 Thus it follows that
$
\sigma_{n,t}^{2}=\mu\left((\xi_{n}(t))^{2}\right) = \sigma_{t}^{2}+ En^{-\eta''},
$
where  
\beqn|E|<C=C(\lambda,A_{*},C_{H},\eta,\eta'',L_{0}, \|f\|_{\operatorname{Lip}}).
\eeqn
\end{proof}

\subsection{Proofs of Theorems \ref{univariate quasitheorem} and \ref{univariate quasitheorem 2}.}

An upper bound on the Wasserstein distance of two normal distributions is given by the next lemma.
\begin{lem}\label{normal distributions distance}
Let $a,b\ge 0$ and $Z\sim N(0,1)$. Then $d_\mathscr{W}(aZ,bZ)\le \dfrac{\sqrt{2}|a-b|}{\sqrt{\pi}}$.
\end{lem}
\begin{proof}
Let $h$ be $1$-Lipschitz and $a,b\ge 0$. Then
\begin{align*}
\left|\bE[h(aZ)]-\bE[h(bZ)]\right| &\le |a-b| \bE\left|Z\right|=\dfrac{\sqrt{2}|a-b|}{\sqrt{\pi}}.
\end{align*}
\end{proof}

Next theorem proves Theorem \ref{univariate quasitheorem} for large values of $n$. For small $n$ Theorem \ref{univariate quasitheorem} holds trivially by choosing large enough $C$. 
\begin{thm}
Let $t_{0}\in \,]0,1]$ and $\gamma$ be such that $\hat{\sigma}_{t_{0}}^{2}(f)>0$. Then for all $\eta'\le \eta$ there exists a constant  $C=C(\lambda, A_{*}, C_{H},\eta, \eta',L_{0}, \|f\|_{\operatorname{Lip}}, t_{0},\hat{\sigma}_{t_{0}}^{2})>0$ and a constant $n_{0}>0$ such that for every $t\ge t_{0}$ and $n\ge n_{0}$
\beqn
d_{\mathscr{W}}\left(\xi_{n}(t), \sigma_{t}Z\right) 
 \le  C(n^{-\frac{1}{2}}\log n+n^{-\eta'}).
\eeqn
\label{pre quasitheorem}
\end{thm}

\begin{proof}
This proof is divided in three steps. The Wasserstein distances
\beq
d_{\mathscr{W}}\left(\xi_{n}\left( t \right),\xi_{n}\left(\lceil nt \rceil/n\right)\right),\quad d_{\mathscr{W}}\left(\xi_{n}\left(\lceil nt \rceil/n\right),\sigma_{\lceil nt \rceil/n}Z\right)\quad \text{and} \quad d_{\mathscr{W}}\left(\sigma_{\lceil nt \rceil/n}Z,\sigma_{t}Z\right)\label{three distances}
\eeq
are estimated in the corresponding order. The final result then follows by triangle inequality.

 Before computing upper bounds on the Wasserstein distances in \eqref{three distances} we need to guarantee that for every  $t\ge t_{0}$ and large enough $n$ the variances $\sigma_{t}$ and $\sigma_{n,t}$ are greater than some constant.
 
Since $t\mapsto \hat{\sigma}_{t}^{2}$ is Hölder continuous by Lemma \ref{Hölder continuity lemma} it follows that there exists $t_{1}=t_{1}(\lambda, A_{*},C_{H},\|f\|_{\operatorname{Lip}},\eta, \hat{\sigma}_{t_{0}}^{2}, t_{0}) \le t_{0}$ such that for every $t\in [t_{1},t_{0}]$ it holds that $\hat{\sigma}_{t}^{2}\ge \dfrac{\hat{\sigma}_{t_{0}}^{2}}{2}$. For $t\ge t_{0}$ this implies $\sigma^{2}_{t}\ge \dfrac{(t_{0}-t_{1})\hat{\sigma}_{t_{0}}^{2}}{2}$ and by Lemma \ref{variance distance lemma} for every $\eta'<\eta$ there exists $C=C(\lambda,A_{*},C_{H},\eta,\eta',L_{0}, \|f\|_{\operatorname{Lip}})$ such that
\beqn
|\sigma_{n,t}^{2}- \sigma_{t}^{2}| \le Cn^{-\eta'}  \quad \Rightarrow  \quad \sigma_{n,t}^{2}\ge \sigma_{t}^{2}-Cn^{-\eta'}\ge \dfrac{(t_{0}-t_{1})\hat{\sigma}_{t_{0}}^{2}}{2}- Cn^{-\eta'}. 
\eeqn
Thus there exists $n_{0}=n_{0}(\lambda,A_{*},C_{H},\eta,\eta',L_{0}, \|f\|_{\operatorname{Lip}},\hat{\sigma}_{t_{0}}^{2})$ such that 
\beq
\sigma_{t}^{2}\ge \dfrac{(t_{0}-t_{1})\hat{\sigma}_{t_{0}}^{2}}{2} \quad \text{and} \quad \sigma_{n,t}^{2}\ge \dfrac{(t_{0}-t_{1})\hat{\sigma}_{t_{0}}^{2}}{4}, \label{positive sigmas}
\eeq
 when $n\ge n_{0}$ and $t\ge t_{0}$.

 To be able to apply Theorem \ref{circle main thm} we also assume that $ n_{0}t_{0}\ge 16/(1-\vartheta)^{2}$. 

\textbf{Step 1.}
Notice that
$
\xi_{n}\left(\lceil nt \rceil/n\right)= \frac{1}{\sqrt{n}}\sum_{i=0}^{\lceil nt \rceil -1} \bar{f}_{n,i}.
$
Thus 
\beq
d_\mathscr{W}\left(\xi_{n}(t),\xi_{n}\left(\frac{\lceil nt \rceil}{n}\right)\right) \le
\left\|\xi_{n}\left(\frac{\lceil nt \rceil}{n}\right)-\xi_{n}(t)\right\|_{\infty}= \left\|\frac{1-\{ nt\}}{\sqrt{n}} \bar{f}_{n,\lfloor nt\rfloor}\right\|_{\infty}\le \dfrac{2\|f\|_{\infty}}{\sqrt{n}}, \label{Wasser 1}
\eeq
where the first inequality follows easily from the definition of Wasserstein distance.

\textbf{Step 2.}
Let $t\ge t_{0}$ and $n\ge n_{0}$. We have by definition
\beqn
\xi_{n}\left(\frac{\lceil nt \rceil}{n}\right)= \frac{\sqrt{\lceil nt\rceil}}{\sqrt{n}}
\left(\frac{1}{\sqrt{\lceil nt\rceil}}\sum_{i=0}^{\lceil nt \rceil -1} \bar{f}_{n,i}
\right).
\eeqn
Denote 
\beqn
V=V(\lceil nt\rceil)= \frac{1}{\sqrt{\lceil nt\rceil}}\sum_{i=0}^{\lceil nt \rceil -1} \bar{f}_{n,i}=\dfrac{\xi_{n}\left(\frac{\lceil nt \rceil}{n}\right) \sqrt{n}}{ \sqrt{\lceil nt\rceil}}.
\eeqn 
Denote the variance of $V(\lceil nt\rceil)$ by $v_{\lceil nt\rceil}^{2}= \dfrac{n}{\lceil nt\rceil}\sigma_{n,\lceil nt\rceil/n}^{2}$.
Since $\lceil nt \rceil\ge 16/(1-\vartheta)^{2}$, we can apply
 Theorem \ref{circle main thm} to $V(\lceil nt\rceil)$ and it yields 
\begin{align}
\begin{split}d_{\mathscr{W}}\left(\xi_{n}\left( \dfrac{\lceil nt\rceil}{n} \right),\sigma_{n,\lceil nt\rceil/n}Z\right) &= d_{\mathscr{W}}\left(\frac{\sqrt{\lceil nt\rceil}}{\sqrt{n}}V,\sigma_{n,\lceil nt\rceil/n}Z\right)
\\
&=\dfrac{\sqrt{\lceil nt\rceil}}{\sqrt{n}}d_{\mathscr{W}}\left(V,v_{\lceil nt\rceil}Z\right)
\\
&\le C\lceil nt\rceil^{-\frac{1}{2}}\log \lceil nt\rceil\le Ct_{0}^{-\frac{1}{2}} n^{-\frac{1}{2}}\log n=C n^{-\frac{1}{2}}\log n, \label{another variance}
\end{split}
\end{align}
where $C=C(\lambda,A_{*},\|f\|_{\operatorname{Lip}},L_{0})$.

 We have $\lceil nt\rceil/n\ge t$ and thus $\sigma_{\lceil nt\rceil/n}\ge \left((t_{0}-t_{1})\hat{\sigma}_{t_{0}}^{2}/2\right)^{\frac{1}{2}}$. Therefore by Lemma \ref{variance distance lemma}
 \beq
\left|\sigma_{n,\lceil nt\rceil/n}- \sigma_{\lceil nt\rceil/n}\right|= \frac{\left|\sigma^{2}_{n,\lceil nt\rceil/n}- \sigma^{2}_{\lceil nt\rceil/n}\right|}{\sigma_{n,\lceil nt\rceil/n}+ \sigma_{\lceil nt\rceil/n}}\le Cn^{-\eta'} \left(\dfrac{(t_{0}-t_{1})\hat{\sigma}_{t_{0}}^{2}}{2}\right)^{-\frac{1}{2}}=Cn^{-\eta'},\label{more variance}
\eeq
where in the last expression $C=C(\lambda,A_{*},C_{H},\eta,\eta',L_{0}, \|f\|_{\operatorname{Lip}},t_{0},\hat{\sigma}_{t_{0}}^{2})$. Thus by \eqref{another variance} and \eqref{more variance}, Lemma \ref{normal distributions distance} yields
\begin{align}
\begin{split}
d_{\mathscr{W}}\left(\xi_{n}\left(\frac{\lceil nt \rceil}{n}\right),\sigma_{\frac{\lceil nt \rceil}{n}}Z\right) &\le d_{\mathscr{W}}\left(\xi_{n}\left( \frac{\lceil nt\rceil}{n} \right),\sigma_{n,\frac{\lceil nt\rceil}{n}}Z\right) + d_{\mathscr{W}}\left(\sigma_{n,\frac{\lceil nt\rceil}{n}}Z,\sigma_{\frac{\lceil nt\rceil}{n}}Z\right)
\\
&\le C(n^{-\frac{1}{2}}\log n+n^{-\eta'}),\label{Wasser 2}
\end{split}
\end{align}
where $C=C(\lambda,A_{*},C_{H},\eta,\eta',L_{0}, \|f\|_{\operatorname{Lip}},t_{0},\hat{\sigma}_{t_{0}}^{2})$.

\textbf{Step 3.} 
By Lemma \ref{Hölder continuity lemma} $t\mapsto \hat{\sigma}_{t}^{2}$ is Hölder continuous and thus $\|\hat{\sigma}_{t}^{2}\|_{\infty}\le C$, for every $t\in [0,1]$, where $C=C(\lambda, A_{*},\|f\|_{\operatorname{Lip}},\eta)$. Therefore $\left|\sigma_{\lceil nt \rceil/n}^{2}- \sigma_{t}^{2}\right| \le Cn^{-1}$. Let $t\ge t_{0}$ and $n\ge n_{0}$. Now by \eqref{positive sigmas}
\beqn
\left|\sigma_{\lceil nt \rceil/n}- \sigma_{t}\right|\le \frac{\left|\sigma_{\lceil nt \rceil/n}^{2}- \sigma_{t}^{2}\right|}{\sigma_{\lceil nt \rceil/n}+ \sigma_{t}}\le Cn^{-1}\left(\dfrac{(t_{0}-t_{1})\hat{\sigma}_{t_{0}}^{2}}{2}\right)^{-\frac{1}{2}}\le Cn^{-1},
\eeqn 
where in the last expression $C=C(\lambda, A_{*},C_{H},\eta,\|f\|_{\operatorname{Lip}},t_{0},\hat{\sigma}_{t_{0}}^{2})$.
 Thus by Lemma \ref{normal distributions distance}
\beq
d_\mathscr{W}\left(\sigma_{\lceil nt \rceil/n}Z,\sigma_{t}Z\right)\le Cn^{-1}.\label{Wasser 3}
\eeq

Collecting the estimates from \eqref{Wasser 1}, \eqref{Wasser 2} and \eqref{Wasser 3}, we see that for $n\ge n_{0}$ and $t\ge t_{0}$
\beqn
d_{\mathscr{W}}\left(\xi_{n}(t), \sigma_{t}Z\right) 
 \le C(n^{-\eta'}+n^{-\frac{1}{2}}\log n),
\eeqn
where $C=C(\lambda, A_{*},C_{H},\eta,\|f\|_{\operatorname{Lip}},L_{0},t_{0},\hat{\sigma}_{t_{0}}^{2})$.
\end{proof}

Next we give the proof of Theorem \ref{univariate quasitheorem 2}:

 Let $0<\eta'<\eta$. Let $n\ge 1$ and $t\in \,[0,1]$. Then at least one of the following cases holds: \textbf{Case 1:} ~$nt\ge 16/(1-\vartheta)^{2} $; \textbf{Case 2:} $t\le n^{-\eta'}$; \textbf{Case 3:} $n\le (16/(1-\vartheta)^{2} )^{1/(1-\eta')}$. We show that in each of these cases, there exists $C=C(\lambda, A_*,\eta',\eta, C_{H}, \|f\|_{\operatorname{Lip}}, L_{0})$ such that
\beq
d_{\mathscr{W}}\left(\xi_{n}(t), \sigma_{t}Z\right) 
 \le Cn^{-\frac{\eta'}{2}}+Cn^{-\frac{1}{6}}\log n. \label{goal bound}
\eeq 

In Case 1, we follow the proof of Theorem \ref{pre quasitheorem} making the following changes. First, we do not define any $t_{0}$ or $t_{1}$. Second, in \eqref{another variance} instead of Theorem \ref{circle main thm}, we apply Corollary \ref{N^1/6 convergence cor}, which yields the bound $Cn^{-1/6}\log n$ on $d_{\mathscr{W}}\left(\xi_{n}\left( \lceil nt\rceil/n \right),\sigma_{n,\lceil nt\rceil/n}Z\right)$. Third,
in estimating 
$
\left|\sigma_{n,\lceil nt\rceil/n}- \sigma_{\lceil nt\rceil/n}\right|$ and $\left|\sigma_{\lceil nt \rceil/n}- \sigma_{t}\right| 
$
we use that for $x_{1},x_{2}\ge 0$ we have $|x_{1}-x_{2}|\le \sqrt{|x_{1}^{2}-x_{2}^{2}|}$. This yields the estimates $\left|\sigma_{n,\lceil nt\rceil/n}Z- \sigma_{\lceil nt\rceil/n}Z\right| \le Cn^{-\frac{\eta'}{2}}$ and $\left|\sigma_{\lceil nt \rceil/n}Z- \sigma_{t}Z\right| \le Cn^{-\frac{1}{2}}$. By \eqref{Wasser 1} we have $d_\mathscr{W}\left(\xi_{n}(t),\xi_{n}\left(\lceil nt \rceil/n\right)\right) \le Cn^{-1/2}$. Overall, collecting these estimates, we have that \eqref{goal bound} holds in Case 1.

In Case 2 we see that $\sigma_{n,nt}\le \|\bar{f}\|_{\infty}t\le Cn^{-\eta'/2}$. Furthermore since $\hat{\sigma}_{s}^{2}$ is bounded we also have $\sigma_{t}\le Ct^{1/2}\le Cn^{-\eta'/2}$. Now \eqref{upper bound by variance} yields  $d_{\mathscr{W}}\left(\xi_{n}(t), \sigma_{t}Z\right) 
 \le Cn^{-\eta'/2}$. Clearly in Case 3  we can choose large enough $C$ such that \eqref{goal bound} holds.

We can now choose $C$ in Theorem \ref{univariate quasitheorem 2} to be the maximum of the corresponding constants in Cases 1--3. This completes the proof of Theorem \ref{univariate quasitheorem 2}.

\subsection{Proof of Theorem \ref{multivariate quasitheorem}}
In this subsection we present the proof of Theorem \ref{multivariate quasitheorem}. 

Let $M$ be a $d\times d$ matrix. We introduce the following norm that is used through this subsection:
\beqn
|M|=\max\{|M_{\alpha\beta}|: \alpha,\beta=1,...,d\}.
\eeqn

The following two lemmas generalize Lemmas \ref{Hölder continuity lemma} and \ref{variance distance lemma}. The proofs are similar and thus omitted.  

\begin{lem}
$t\rightarrow (\hat{\Sigma}_{t}^{2})_{\alpha\beta}(f)$ is Hölder continuous with every exponent $\eta''<\eta$ and entries $\alpha,\beta\in \{1,...,d\}$. An upper bound for the corresponding Hölder constant $C$ can be given as a function of $\lambda, A_{*},\eta,\eta'', C_{H}$ and $\|f\|_{\operatorname{Lip}}$. \label{Hölder continuity lemma multiv}
\end{lem}

\begin{lem}
Let $\eta''<\eta$. Then there exists a constant $C=C(\lambda,A_{*},C_{H},\eta,\eta'',L_{0}, \|f\|_{\operatorname{Lip}})$ such that
\beqn
\left|\Sigma_{n,t}-\Sigma_{t}\right| \le Cn^{-\eta''}
\eeqn
for every $t\in [0,1]$.
\label{covariance distance lemma}
\end{lem}

The upper bound on
$
\left|\mu(h(\xi_{n}(t)))-\Phi_{\Sigma_{t}}(h)\right|
$
is found by bounding the following four terms: 
\beq
\left|\mu\left[h(\xi_{n}(t))\right]- \mu\left[h\left(\xi_{n}\left(\lceil nt\rceil/n\right)\right)\right] \right|,\label{eq:multiv 1}
\eeq

\beq
\left|\mu\left[h\left(\xi_{n}\left(\lceil nt\rceil/n\right)\right)\right] - \Phi_{\Sigma_{n,\lceil nt\rceil/n}}(h)\right|,\label{eq:multiv 2}
\eeq

\beq
\left|\Phi_{\Sigma_{n,\lceil nt\rceil/n}}(h)-\Phi_{\Sigma_{\lceil nt\rceil/n}}(h)\right|,\label{eq:multiv 3}
\eeq
and 
\beq
\left|\Phi_{\Sigma_{\lceil nt\rceil/n}}(h)-\Phi_{\Sigma_{t}}(h)\right|.\label{eq:multiv 4}
\eeq
The proof is analogous to the proof of Theorem \ref{pre quasitheorem}. Bounding \eqref{eq:multiv 1} corresponds to Step 1, \eqref{eq:multiv 2} and \eqref{eq:multiv 3} to Step 2, and \eqref{eq:multiv 4} to Step 3. However, computing with matrices introduces some complications that need to be dealt with more closely. We therefore first introduce some matrix-related notations. 

If a matrix $M$ is positive semi-definite, we denote $M\ge 0$, and if it is positive definite, $M>0$. The minimal eigenvalue of $M$ is denoted by $\lambda_{1}(M)$. Recall that $\lambda_{1}(M)\ge 0$ or $\lambda_{1}(M)>0$ if and only if $M\ge 0$ or $M>0$, respectively.

Let $v\in \bR^{d}$. We denote its Euclidean norm by $|v|_{d}$ and for a $d\times d$ matrix $M$, we write
\beq
\|M\|=\sup_{v\in \bR^{d}\setminus \{0\}}\frac{|Mv|_{d}}{|v|_{d}}\le d|M| \label{spectral norm}
\eeq
and call $\|M\|$ the spectral norm of $M$.

\textbf{Matrix fact.} Let $M$ be a positive definite matrix satisfying $\lambda_{1}(M)\ge C_{l}$ and $\|M\|\le C_{u}$ for some $0<C_{l}<C_{u}$. Then there exists $\delta= \delta(C_{l},C_{u})>0$ such that if $\tilde{M}$ is positive semi-definite and $|M-\tilde{M}|<\delta$, then $\lambda_{1}(\tilde{M})\ge C_{l}/2$ and especially $\tilde{M}$ is positive definite.

Let now $t_{0}\in \,]0,1]$ be such that $\hat{\Sigma}_{t_{0}}>0$. By Lemma \ref{Hölder continuity lemma multiv}, the entries of $\hat{\Sigma}_{t}$ vary Hölder continuously. Thus there exists a neighbourhood of $t_{0}$ such that if $t$ is in that neighbourhood, then $\hat{\Sigma}_{t}>0$. For all $t\in [0,1]$, we also have $\hat{\Sigma}_{t}\ge 0$, since covariance matrices are positive semi-definite. This guarantees that $\Sigma_{t}=\int_{0}^{t}\hat{\Sigma}_{s}ds>0$ for every $t\ge t_{0}$. To be more precise 
\beq
\lambda_{1}(\Sigma_{t})= \lambda_{1}\left(\Sigma_{t_{0}}+\int_{t_{0}}^{t}\hat{\Sigma}_{s}ds\right)\ge  \lambda_{1}(\Sigma_{t_{0}}) +\lambda_{1}\left(\int_{t_{0}}^{t}\hat{\Sigma}_{s}ds\right)\ge \lambda_{1}(\Sigma_{t_{0}}), \label{lambda eq}
\eeq 
when $t\ge t_{0}$.
\\   
\textbf{Bound on \eqref{eq:multiv 1}.}
As in the Step 1 of the proof of Theorem \ref{pre quasitheorem}, we have that \eqref{eq:multiv 1} is bounded by $Cn^{-\frac{1}{2}}$. 
\\
\textbf{Bound on \eqref{eq:multiv 2}.}
We are going to apply Theorem \ref{multivariate circle theorem} as in Step 2 of Theorem \ref{pre quasitheorem}. To this end  $\Sigma_{n,t}$ must be positive definite. As will be apparent later it is crucial that we can choose $n_{0}$ independent of $t$ such that for every $n\ge n_{0}$ and $t\ge t_{0}$ we have $\Sigma_{n,t}>0$.

By \eqref{lambda eq}, in the set $\{\Sigma_{t}:t\ge t_{0}\}$ there exists an uniform bound to $\lambda_{1}(\Sigma_{t})$ namely $\lambda_{1}(\Sigma_{t_{0}})$. By Lemma \ref{Hölder continuity lemma multiv} the entries of $\Sigma_{t}$ are also uniformly bounded. Thus by Matrix fact and Lemma \ref{covariance distance lemma} there exists $n_{0}$ independent of $t$ such that $\Sigma_{n,t}>0$ for every $n\ge n_{0},t\ge t_{0}$. Choose $n_{0}$ such that also $ n_{0}t_{0} \ge 16/(1-\vartheta)^{2}$. We will bound \eqref{eq:multiv 2} by first applying Theorem \ref{multivariate circle theorem} to $\frac{1}{\sqrt{\lceil nt\rceil}}\sum_{i=0}^{\lceil nt \rceil -1} \bar{f}_{n,i}$. Denote
\beqn
V=V(\lceil nt\rceil)= \frac{1}{\sqrt{\lceil nt\rceil}}\sum_{i=0}^{\lceil nt \rceil -1} \bar{f}_{n,i}=\dfrac{\sqrt{n}}{ \sqrt{\lceil nt\rceil}}\xi_{n}\left(\frac{\lceil nt \rceil}{n}\right).
\eeqn
Now the covariance matrix $\Sigma_{V}$ of $V(\lceil nt\rceil)$ is $\dfrac{n}{\lceil nt\rceil}\Sigma_{n,\lceil nt\rceil/n}$ and thus positive definite, when $n\ge n_{0}$ and $t\ge t_{0}$.

Let $h\colon \bR^{d}\rightarrow \bR$ be as in the theorem. Define $h^{*}\colon \bR^{d}\rightarrow \bR$, $w\mapsto h\left(\dfrac{\sqrt{\lceil nt\rceil}}{\sqrt{n}}w\right)$. Thus we have $\mu(h^{*}(V))=\mu(h(\xi_{n}(\lceil nt\rceil/n)))$ and $\| D^{k}h^{*}\|_{\infty}\le \| D^{k}h\|_{\infty}<\infty$. We have
$
 \Phi_{\Sigma_{n,\lceil nt\rceil/n}}(h) 
=  \Phi_{\Sigma_{V}}(h^{*}).
$
Since $\lceil nt \rceil\ge 16/(1-\vartheta)^{2} $, we can apply
 Theorem \ref{multivariate circle theorem} to $V(\lceil nt \rceil)$ and it yields 

\begin{align*}
\left| \mu(h(\xi_{n}(\lceil nt\rceil/n))) - \Phi_{\Sigma_{n,\lceil nt\rceil/n}}(h) \right|
=\left|\mu(h^{*}(V)) - \Phi_{\Sigma_{V}}(h^{*})\right| \le Cn^{-\frac{1}{2}}\log n.
\end{align*}

\textbf{Bound on \eqref{eq:multiv 3}.}
Let $Z\sim \cN(0,I_{d\times d})$, where $\cN(0,I_{d\times d})$ is a standard $d$-dimensional normal distribution. 
If $\Sigma$ is a positive definite $d\times d$ matrix, then it has unique positive definite square root matrix $\Sigma^{1/2}$. Furthermore
\beq
 \Phi_{\Sigma}(h) = \bE[h(\Sigma^{1/2}Z)]. \label{square root matrix}
\eeq
 We define
\beqn
\operatorname{Lip}_{d}(h)=\sup_{x,y\in\bR^d, x\neq y}\frac{|h(x)-h(y)|}{|x-y|_{d}}.
\eeqn
\noindent With these definitions
\begin{align}
\begin{split}
\left|\bE\left[h\left(\Sigma_{n,\lceil nt\rceil/n}^{1/2}Z\right)\right] - \bE\left[h\left(\Sigma_{\lceil nt\rceil/n}^{1/2}Z\right)\right]\right| &\le \operatorname{Lip}_{d}(h) \bE\left|\Sigma_{n,\lceil nt\rceil/n}^{1/2}Z-\Sigma_{\lceil nt\rceil/n}^{1/2}Z\right|_{d} \label{spectral bound}
\\
&\le \operatorname{Lip}_{d}(h)\bE|Z|_{d}\, \left\|\Sigma_{n,\lceil nt\rceil/n}^{1/2}-\Sigma_{\lceil nt\rceil/n}^{1/2}\right\|.\end{split}
\end{align}
The following bound holds for the spectral norm of the difference of two square root matrices (see \cite{Schmitt_1992})
\beq
\left\|\Sigma_{n,\lceil nt\rceil/n}^{1/2}-\Sigma_{\lceil nt\rceil/n}^{1/2}\right\| \le \frac{ \left\|\Sigma_{n,\lceil nt\rceil/n}-\Sigma_{\lceil nt\rceil/n}\right\|}{\sqrt{\lambda_1\left(\Sigma_{n,\lceil nt\rceil/n}\right)}+\sqrt{\lambda_1\left(\Sigma_{\lceil nt\rceil/n}\right)}}.\label{eq:square root bound}
\eeq
Now \eqref{square root matrix}, \eqref{spectral bound}, \eqref{eq:square root bound} and \eqref{spectral norm} yield 
\beq
\left|\Phi_{\Sigma_{n,\lceil nt\rceil/n}}(h)-\Phi_{\Sigma_{\lceil nt\rceil/n}}(h)\right|\le  \operatorname{Lip}_{d}(h)\bE|Z|_{d}\frac{ d\left|\Sigma_{n,\lceil nt\rceil/n}-\Sigma_{\lceil nt\rceil/n}\right|}{\sqrt{\lambda_1\left(\Sigma_{n,\lceil nt\rceil/n}\right)}+\sqrt{\lambda_1\left(\Sigma_{\lceil nt\rceil/n}\right)}}. \label{almost final estimate}
\eeq

Let $\eta''< \eta$. By Lemma \ref{covariance distance lemma} we have $\left|\Sigma_{n,\frac{\lceil nt\rceil}{n}}-\Sigma_{\frac{\lceil nt\rceil}{n}}\right|\le Cn^{-\eta''}$. The other terms on the right side of inequality \eqref{almost final estimate} are uniformly bounded. Thus 
\beqn
\left|\Phi_{\Sigma_{n,\lceil nt\rceil/n}}(h)-\Phi_{\Sigma_{\lceil nt\rceil/n}}(h)\right|\le Cn^{-\eta''}.
\eeqn
\\
\textbf{Bound on \eqref{eq:multiv 4}.} Following the same steps as in the previous calculation, we have 
\beq
\left|\Phi_{\Sigma_{\lceil nt\rceil/n}}(h)-\Phi_{\Sigma_{t}}(h)\right|\le  \operatorname{Lip}_{d}(h)\bE|Z|_{d}\frac{ d\left|\Sigma_{\lceil nt\rceil/n}-\Sigma_{t}\right|}{\sqrt{\lambda_1\left(\Sigma_{\lceil nt\rceil/n}\right)}+\sqrt{\lambda_1\left(\Sigma_{t}\right)}}. \label{almost final estimate2}
\eeq
Since $|\hat{\Sigma}_{t}|$ is clearly uniformly bounded, we have the uniform estimate
\beqn
\left|\Sigma_{\lceil nt\rceil/n}-\Sigma_{t}\right|=\left|\int_{t}^{\lceil nt\rceil/n}\hat{\Sigma}_{s}ds\right| \le Cn^{-1}.
\eeqn
This and \eqref{almost final estimate2} yields the bound
\beqn
\left|\Phi_{\Sigma_{\lceil nt\rceil/n}}(h)-\Phi_{\Sigma_{t}}(h)\right|\le Cn^{-1}.
\eeqn

\textbf{Bound on $\left|\mu(h(\xi_{n}(t)))-\Phi_{\Sigma_{t}}(h)\right|$.} There exist $n_{0}$ such that for all $n\ge n_{0}$ the bounds of \eqref{eq:multiv 1}, \eqref{eq:multiv 2}, \eqref{eq:multiv 3} and \eqref{eq:multiv 4} computed above hold and thus we have 
\beq
\left|\mu(h(\xi_{n}(t)))-\Phi_{\Sigma_{t}}(h)\right|\le Cn^{-\frac{1}{2}}+ Cn^{-\frac{1}{2}}\log n +Cn^{-\eta''}+Cn^{-1}\le Cn^{-\frac{1}{2}}\log n +Cn^{-\eta''}\label{final bound}
\eeq
for every $n\ge n_{0}$.
It is easy to choose large enough $C$ so that \eqref{final bound} holds also when $1\le n\le n_{0}$.
This completes the proof of Theorem \ref{multivariate quasitheorem}. \qed

\section{Proofs of abstract results}\label{main proofs}
This section assumes familiarity with \cite{HellaLeppanenStenlund_2016}. However, accepting certain results given, we have made an effort to provide a comprehensible proof.

Recall that the goal of Theorem \ref{thm:main} is to control the term $|\mu(h(W)) - \Phi_{\Sigma_{N}}(h)|$.

First, assuming that the matrix~\(\Sigma_{N} \in \bR^{d\times d}\) is positive definite, the normal distribution $\cN(0,\Sigma_{N})$ has a density function $\phi_{\Sigma_{N}}$, and we define
\begin{align*}
&A(w)= -\int_0^{\infty} \! \left\{\int_{\bR^d} h(e^{-s}w + \sqrt{1 - e^{-2s}}\,z)\,\phi_{\Sigma_{N}}(z)\,dz - \Phi_{\Sigma_{N}}(h)\right\} \, ds,
\end{align*}
where $\Phi_{\Sigma_{N}}(h)$ stands for the expectation of $h$ with respect to the same normal distribution.
Furthermore, if  \(h: \, \bR^d \to \bR\) is three times differentiable with \(\Vert D^k h \Vert_{\infty} < \infty\) for \(1 \le k \le 3\), then \(A \in C^3(\bR^d,\bR)\), and it solves the Stein equation
\begin{align}\label{eq:steinmv}
\tr  \Sigma D^2A(w) - w \cdot  \nabla A(w) = h(w) - \Phi_{\Sigma}(h)
\end{align}
We refer to \cite{Barbour_1990,Gotze_1991,GoldsteinRinott_1996, Gaunt_2016} or Lemma 3.3 in \cite{HellaLeppanenStenlund_2016}. Moreover
$
\| \partial_1^{t_1}\cdots \partial_d^{t_d} A \|_\infty \le k^{-1} \| \partial_1^{t_1}\cdots \partial_d^{t_d} h \|_\infty
$
whenever $t_1+\cdots+t_d = k$, $1\le k\le 3$.

Thus in \eqref{eq:steinmv}, taking the expected value with respect to $\mu$, we have
\beqn
|\mu [h(W)] - \Phi_{\Sigma}(h)| = |\mu[\tr  \Sigma D^2A(W) - W \cdot  \nabla A(W)]|.
\eeqn
It turns out that the expression on the right side is easier to bound than the left, which is the core of Stein's method.

Instead of proving Theorem \ref{thm:main} directly, we prove the following preliminary result.
\begin{thm}\label{thm:pre} Let $(X,\cB,\mu)$ be a probability space and $(f^{i})_{i=0}^{\infty}$ a sequence of random vectors with common upper bound $\|f\|_{\infty}$. Let $A\in C^3(\bR^d,\bR)$ be a given function satisfying \(\| D^k A \|_{\infty} < \infty\) for \(1 \le k \le 3\). Fix integers $N>0$ and $0\le K< N$.  Suppose that the following conditions are satisfied:

\begin{itemize}

\item[(A1)] There exist constants $C_2 > 0$ and $C_4 > 0$, and a non-increasing function \(\rho : \, \bN_0 \to \bR_+\) with \(\rho(0) = 1\) and \( \sum_{i=1}^{\infty} i\rho(i) < \infty\), such that for all \(0\le i \le j \le k \le l \le N-1\),
\begin{align*}
|\mu(\bar{f}^{i}_{\alpha}\bar{f}^{j}_{\beta})| \leq C_2\rho(j-i),
\end{align*}
\begin{align*}
&|\mu(\bar{f}^{i}_{\alpha}\bar{f}^{j}_{\beta}\bar{f}^{k}_{\gamma}\bar{f}^{l}_{\delta})| \le C_4\rho(\max\{j-i,l-k\}), \\
&|\mu(\bar{f}^{i}_{\alpha}\bar{f}^{j}_{\beta}\bar{f}^{k}_{\gamma}\bar{f}^{l}_{\delta}) - \mu(\bar{f}^{i}_{\alpha}\bar{f}^{j}_{\beta})\mu(\bar{f}^{k}_{\gamma}\bar{f}^{l}_{\delta})| \le C_4\rho(k-j).
\end{align*}
hold whenever $k\ge 0$; $0\le i\le j \le k \le l < N$; $\alpha,\beta,\gamma,\delta\in\{\alpha',\beta'\}$ and $\alpha',\beta'\in\{1,\dots,d\}$.

\smallskip
\item[(A2')] There exists a function \(\eta: \, \bN_0^2 \to \bR_+\) such that
\begin{align*}
\left|\mu\!\left(\frac{1}{\sqrt{N}}\sum_{n=0}^{N-1} \bar{f}^n \cdot  \nabla A(W_{n})\right)\right|  \leq \eta(N,K).
\end{align*}


\end{itemize}
Then
\begin{align*}
&|\mu(\tr \Sigma_{N} D^2A(W) - W \cdot  \nabla A(W))| \\
&\qquad \le 2d^3 C_2\Vert f \Vert_{\infty} \Vert D^3 A \Vert_{\infty}\frac{2K+1}{\sqrt{N}} \left(\rho(0) + 2\sum_{i = 1}^{2K} \rho(i)\right)
 \\
&\qquad\quad+ 2d^2C_2\Vert D^2A\Vert_\infty \sum_{i=K+1}^{N-1}\rho(i)
\\
&\qquad\quad+ 11d^2\max\{C_2,\sqrt{C_4}\} \Vert D^2 A\Vert_{\infty}\frac{\sqrt{K+1}}{\sqrt{N}}\sqrt{\sum_{i=0}^{N-1} (i+1)\rho(i)}\, \\
&\qquad\quad+ \eta(N,K).
\end{align*}
\end{thm}

\subsection{Proof of Theorem~\ref{thm:pre}}\label{sec:proof_pre}
Let the Assumptions of Theorem \ref{thm:pre} hold. By basic matrix computations we see that $\mu(\tr \Sigma_{N} D^2A(W) - W \cdot  \nabla A(W))$ can be represented as a sum

\begin{align}
& \mu\!\left( \frac{1}{\sqrt{N}} \sum_{n=0}^{N-1}  D^2A(W)(W-W_{n})  -  \bar{f}^n \cdot (  \nabla A(W) -  \nabla A(W_{n})) \right)  \label{eq1}\\
+\, &\mu\!\left( \tr \left(\left(  \Sigma_{N} - \frac{1}{N}\sum_{n=0}^{N-1}\sum_{m \in [n]_K} \bar{f}^n \otimes \bar{f}^m \right)\! D^2A(W) \right)\right) \label{eq3}\\
+\, &\mu\!\left(\frac{1}{\sqrt{N}}\sum_{n=0}^{N-1}-\bar{f}^n \cdot  \nabla A(W_{n})\right)  \label{eq2}.
\end{align}
By Assumption~(A2'), the absolute value of \eqref{eq2} is bounded by $\eta(N,K)$. Bounds for the absolute values of \eqref{eq1} and \eqref{eq3} are stated in Propositions ~\ref{prop1} and ~\ref{prop3}, respectively.

Proposition \ref{prop1} is proved exactly as Proposition 4.3 in \cite{HellaLeppanenStenlund_2016} and the proof is thus omitted. The only difference is that every $\|\bar{f}^{i}\|_{\infty}$ is bounded above by $2 \|f\|_{\infty}$ which explains the coefficient $2$ in the bound.
\begin{prop}\label{prop1}
The absolute value of expression in~\eqref{eq1} is bounded by
\begin{align*}
2d^3 C_2\Vert f \Vert_{\infty} \Vert D^3 A \Vert_{\infty}\frac{2K+1}{\sqrt{N}} \left(\rho(0) + 2\sum_{i = 1}^{2K} \rho(i) \right).
\end{align*}
 \end{prop}

The next proposition gives a bound on the absolute value of expression ~\eqref{eq3}.
\begin{prop}\label{prop3}
The absolute value of the expression in~\eqref{eq3} is bounded by
\beqn
\begin{split}
& 2 d^2C_2 \Vert D^2 A \Vert_{\infty}  \sum_{i = K+1}^{N-1}  \rho(i)
\\
+\ &
11d^2\max\{C_2,\sqrt{C_4}\} \| D^2 A\|_{\infty}\frac{\sqrt{K+1}}{\sqrt{N}}\sqrt{\sum_{i=0}^{N-1} (i+1)\rho(i)}.
\end{split}
\eeqn
\end{prop}
To prove Proposition ~\ref{prop3}, we first define
$
\widetilde\Sigma = \frac{1}{N}\sum_{n=0}^{N-1}\sum_{m \in [n]_K} \mu(\bar{f}^n \otimes \bar{f}^m).
$
Using this definition, we have the following upper bound on the absolute value of \eqref{eq3} 
\begin{align}
& \mu\!\left(\left| \tr \! \left(\left(  \widetilde\Sigma - \frac{1}{N}\sum_{n=0}^{N-1}\sum_{m \in [n]_K} \bar{f}^n \otimes \bar{f}^m \right)\! D^2A(W) \right)\right|\right) \label{eq4} \\
& + \| { \tr (( \Sigma_{N} - \widetilde\Sigma ) D^2A ) } \|_\infty. \label{eq5}
\end{align}

Bounds on \eqref{eq4} and \eqref{eq5} are given in the Lemmas \ref{lem4} and \ref{lem5}, respectively. 

The next lemma is proven exactly as Lemma 4.6 in \cite{HellaLeppanenStenlund_2016}. 
\begin{lem}\label{lem4}
The expression ~\eqref{eq4} satisfies the bound
\begin{align*}
& \mu\!\left(\left| \tr \! \left(\left(  \widetilde\Sigma - \frac{1}{N}\sum_{n=0}^{N-1}\sum_{m \in [n]_K} \bar{f}^n \otimes \bar{f}^m \right)\! D^2A(W) \right)\right|\right)
\\
\le\ &11d^2\max\{C_2,\sqrt{C_4}\}\| D^2 A\|_{\infty}\frac{\sqrt{K+1}}{\sqrt{N}}\sqrt{\sum_{i=0}^{N-1} (i+1)\rho(i)}.
\end{align*}

\end{lem}

\begin{lem}\label{lem5}
The expression in~\eqref{eq5} satisfies the bound
\begin{align*}
 \| { \tr (( \Sigma_{N} - \widetilde\Sigma ) D^2A ) }  \|_\infty \le 2d^2C_2\Vert D^2A\Vert_\infty \sum_{i=K+1}^{N-1}\rho(i).
\end{align*} 
\end{lem}

\begin{proof} 
By definitions we have
\begin{align*}
\Sigma_{N} -  \widetilde\Sigma =&\, \frac{1}{N}\sum_{n=0}^{N-1}\sum_{m=0}^{N-1}\mu(\bar{f}^n \otimes \bar{f}^m) -\frac{1}{N}\sum_{n=0}^{N-1}\sum_{m \in [n]_K} \mu(\bar{f}^n \otimes \bar{f}^m)
\\
=&\, \frac{1}{N}\sum_{n=0}^{N-1}\sum_{0\le m \le N-1, m \notin [n]_K} \mu(\bar{f}^n \otimes \bar{f}^m).
\end{align*}
Assumption~(A1) yields
\begin{align*}
|(\Sigma_{N} -  \widetilde\Sigma)_{\alpha\beta}| &= \frac{1}{N}\left|\sum_{n=0}^{N-1}\sum_{0\le m \le N-1, m \notin [n]_K} \mu(\bar{f}^n_\alpha \bar{f}^m_\beta)\right|
\\
&\le \frac{1}{N}\sum_{n=0}^{N-1}\sum_{0\le m \le N-1, m \notin [n]_K} C_{2}\rho(|n-m|)
\\
&\le 2C_{2}\sum_{i=K+1}^{N-1}\rho(i).
\end{align*}
Thus
\begin{align*}
&\| {\tr((\Sigma_{N} -  \widetilde\Sigma )D^2A) }\|_\infty 
 \le  \|D^2A\|_\infty \sum_{1\le\alpha,\beta\le d} |( \Sigma - \widetilde\Sigma)_{\alpha\beta}| \\
\le \,&\|D^2A\|_\infty\sum_{1\le\alpha,\beta\le d}2C_{2}\sum_{i=K+1}^{N-1}\rho(i)
 \\
\le \,&2d^2C_2\Vert D^2A\Vert_\infty \sum_{i=K+1}^{N-1}\rho(i).
\end{align*}
This finishes the proof of the lemma.
\end{proof}

Observe that Proposition~\ref{prop3} follows from Lemmas~\ref{lem4} and~\ref{lem5}. This completes the proof of Proposition~\ref{prop3}.

Proposition~\ref{prop1}, Proposition~\ref{prop3} and Assumption~(A2') now yield the bounds on ~\eqref{eq1}, ~\eqref{eq3} and ~\eqref{eq2}, respectively. This completes the proof of Theorem~\ref{thm:pre}.\qed

\subsection{Proof of Theorem \ref{thm:main}}
To be able to use Theorem \ref{thm:pre} in proving Theorem~\ref{thm:main}, we need to check that Assumption~(A2') is implied by the assumptions of Theorem~\ref{thm:main}. This follows from the next lemma, which is proven with exactly the same steps as Lemma 5.1 in \cite{HellaLeppanenStenlund_2016}. 
\begin{lem}
Under the conditions of Theorem~\ref{thm:main},
\beqn
\left|\mu\!\left(\frac{1}{\sqrt{N}}\sum_{n=0}^{N-1} \bar{f}^n \cdot  \nabla A(W^n)\right)\right|  \leq \sqrt N\tilde\rho(K).
\eeqn 
\end{lem}

Now Assumption (A2') of Theorem~\ref{thm:pre} is satisfied by defining 
\beqn
\eta(N,K) = \sqrt N\tilde\rho(K).
\eeqn

For the purpose of computing the constant $C_{*}$ in the error bound of Theorem \ref{thm:main} we introduce the following lemma.
\begin{lem}
Let $\rho \colon \bN\rightarrow \bR$ be a function such that $0\le\rho(i)\le 1$ for all $i\in \bN$ and $\sum_{i=0}^{\infty}(i+1)\rho(i)< \infty$ Then
$
\left(\sum_{i=0}^{\infty}\rho(i)\right)^{2}\le 2\sum_{i=0}^{\infty}(i+1)\rho(i)
$
\end{lem}\label{sum lemma}
\begin{proof} We have
\beq
\left(\sum_{i=0}^{\infty}\rho(i)\right)^{2}=\sum_{i=0}^{\infty}\sum_{j=0}^{\infty}\rho(i)\rho(j).\label{double sum}
\eeq
 Let $n\in \bN$. In the sum on the right side of \eqref{double sum} there exist exactly $2n+1$ terms $\rho(i)\rho(j)$ such that $\max\{i,j\}=n$. The result now follows easily from rearranging the terms according to $\max\{i,j\}$ and noticing that $\rho(i)\rho(j)\le \rho(\max\{i,j\})$. 
\end{proof}
Taking square roots of the result in Lemma \ref{sum lemma}, we have
\beq
\sum_{i=0}^{\infty}\rho(i) \le\sqrt{2\sum_{i=0}^{\infty}(i+1)\rho(i)}. \label{sqrt of sum}
\eeq

Lemma 3.3 in \cite{HellaLeppanenStenlund_2016} and Theorem~\ref{thm:pre}, followed by elementary estimates and using \eqref{sqrt of sum}, now yield
\begin{align*}
 |\mu(h(W)) - \Phi_{\Sigma_{N}}(h)| 
&=|\mu(\tr \Sigma_{N} D^2A(W) - W \cdot  \nabla A(W))| \\
&\le 2d^3 C_2\Vert f \Vert_{\infty} \Vert D^3 A \Vert_{\infty}\frac{2K+1}{\sqrt{N}} \left(\rho(0) + 2\sum_{i = 1}^{2K} \rho(i)\right)
 \\
&\qquad+ 2d^2C_2\Vert D^2A\Vert_\infty \sum_{i=K+1}^{N-1}\rho(i)
\\
&\qquad+ 11d^2\max\{C_2,\sqrt{C_4}\} \Vert D^2 A\Vert_{\infty}\frac{\sqrt{K+1}}{\sqrt{N}}\sqrt{\sum_{i=0}^{N-1} (i+1)\rho(i)}\, \\
&\qquad+ \eta(N,K).
\\
& \le \frac 83 d^3 C_2\Vert f \Vert_{\infty} \Vert D^3 h \Vert_{\infty}\frac{K+1}{\sqrt{N}}\sum_{i = 0}^{\infty} \rho(i)
+ d^2C_2\Vert D^2 h\Vert_\infty \sum_{i=K+1}^{\infty}\rho(i)
\\
&\qquad+ \frac{11}{2}d^2\max\{C_2,\sqrt{C_4}\} \Vert D^2 h\Vert_{\infty}\frac{K+1}{\sqrt{N}}\sqrt{\sum_{i=0}^{\infty} (i+1)\rho(i)}\, \\
&\qquad+ \sqrt N\tilde\rho(K)
\\
\end{align*}
\begin{align*}
& \le \left(\frac{8\sqrt 2}{3}d^3 C_2\Vert f \Vert_{\infty} \Vert D^3 h \Vert_{\infty}+ 6d^2\max\{C_2,\sqrt{C_4}\} \Vert D^2 h\Vert_{\infty}\right)\frac{K+1}{\sqrt{N}}\sqrt{\sum_{i=0}^{\infty} (i+1)\rho(i)}
 \\
 &\qquad +d^2C_2\Vert D^2 h\Vert_\infty \sum_{i=K+1}^{\infty}\rho(i)
+\sqrt N\tilde\rho(K)
\\
& \le C_*\left(\frac{K+1}{\sqrt{N}}+\sum_{i=K+1}^{\infty}\rho(i)\right) + \sqrt N\tilde\rho(K),
\end{align*}
where
\beqn
C_* = 6d^3\max\{C_2,\sqrt{C_4}\}\left(\Vert f \Vert_{\infty} \Vert D^3 h \Vert_{\infty}+ \Vert D^2 h\Vert_{\infty}\right)\sqrt{\sum_{i=0}^{\infty} (i+1)\rho(i)}.
\eeqn
The proof of Theorem~\ref{thm:main} is now complete. \qed

\subsection{Proof of Theorem \ref{main thm 1d}}
The Stein equation in the univariate case is
\beq
\sigma^2 A'(w) - wA(w) = h(w) - \Phi_{\sigma^2}(h).\label{Stein 1d}
\eeq
The following theorem is proven as Theorem 4.2 in \cite{HellaLeppanenStenlund_2016} with the same modifications as in the proof of Theorem \ref{thm:pre}.

\begin{thm}\label{pre-result 1-D}
Let $(X,\cB,\mu)$ be a probability space and $(f^{i})_{i=0}^{\infty}$ a sequence of random variables with common upper bound $\|f\|_{\infty}$. Let $A\in C^1(\bR,\bR)$ be a given function with absolutely continuous~$A'$, satisfying $\|A^{(k)}\|_\infty < \infty$ for $0\le k\le 2$. Fix integers $N>0$ and $0\le K<N$. Suppose that the following conditions are satisfied:

\begin{itemize}
\item[(B1)] There exist constants \(C_2,C_4\) and a decreasing function \(\rho : \, \bN \to \bR\) with \(\rho(0) = 1\), such that for all \(i \le j \le k \le l\),
\begin{align*}
|\mu(\bar{f}^{i}\bar{f}^{j})| \leq C_2\rho(j-i),
\end{align*}
\begin{align*}
&|\mu(\bar{f}^{i}\bar{f}^{j}\bar{f}^{k}\bar{f}^{l})| \le C_4\rho(\max\{j-i,l-k\}), \\
&|\mu(\bar{f}^{i}\bar{f}^{j}\bar{f}^{k}\bar{f}^{l}) - \mu(\bar{f}^{i}\bar{f}^{j})\mu(\bar{f}^{k}\bar{f}^{l})| \le C_4\rho(k-j).
\end{align*}
\item[(B2')] There exists a function $\eta : \bN_{0}^{2}\rightarrow \bR_{+}$ such that
\begin{align*}
\left|\mu\left[\frac{1}{\sqrt{N}}\sum_{n=0}^{N-1} \bar{f}^{n}A(\bar{W}_n)\right]\right|  \leq  \eta(N,K).
\end{align*}

\end{itemize}
Then 
\begin{align*}
|\mu(\sigma^2_{N} A'(W)-WA(W))|
& \le  C_2\Vert f \Vert_{\infty} \Vert A'' \Vert_{\infty}\frac{2K+1}{\sqrt{N}} \left(\rho(0) + 2\sum_{i = 1}^{2K} \rho(i) \right)
 \\
&\quad+ 2C_2\Vert A' \Vert_{\infty}  \sum_{i= K+1}^{N-1}  \rho(i)  
\\
&\quad+ 11\max\{C_2,\sqrt{C_4}\} \Vert A'\Vert_{\infty}\frac{\sqrt{K+1}}{\sqrt{N}}\sqrt{\sum_{i=0}^{N-1} (i+1)\rho(i)}\, \\
&\quad+ \eta(N,K).
\end{align*}
\end{thm}
Let $\mathscr{F}_{\sigma_{N}^2}$ be the class of differentiable functions $A\colon \bR\rightarrow \bR$ with absolutely continuous derivative and satisfying the following bounds
\beqn
\|A\|_\infty \le 2, 
\quad
\|A'\|_\infty \le \sqrt{2/\pi}\, \sigma^{-1}
\quad\text{and}\quad
\|A''\|_\infty \le 2 \sigma^{-2}.
\eeqn

If $h\colon \bR \rightarrow \bR$ is $1$-Lipschitz, then the corresponding solution $A$ for the Stein equation \eqref{Stein 1d} belongs to $\mathscr{F}_{\sigma_{N}^2}$. 

Following the proof of Theorem 2.3 in \cite{HellaLeppanenStenlund_2016}, we see that if assumptions in Theorem \ref{main thm 1d} are satisfied, then for every $A\in\mathscr{F}_{\sigma_{N}^2}$ the assumptions of Theorem~\ref{pre-result 1-D} are satisfied with the choice
\beqn
\eta(N,K) = 2\max\{1,\sigma_{N}^{-2}\}\sqrt N\tilde\rho(K).
\eeqn
Therefore using the same methods as in the proof of Theorem \ref{thm:main}, we have
\begin{align*}
d_\mathscr{W}(W,\sigma_{N}Z)
&\le |\mu(\sigma^2_{N} A'(W)-WA(W))| \\
&\le  C_2\Vert f \Vert_{\infty} 2\sigma_{N}^{-2}\frac{2K+1}{\sqrt{N}} \left(\rho(0) + 2\sum_{i = 1}^{2K} \rho(i) \right)
+ 2C_2\sqrt{2/\pi}\sigma_{N}^{-1}  \sum_{i= K+1}^{N-1}  \rho(i)  
\\
&\quad+ 11\max\{C_2,\sqrt{C_4}\} \sqrt{2/\pi}\sigma_{N}^{-1}\frac{\sqrt{K+1}}{\sqrt{N}}\sqrt{\sum_{i=0}^{N-1} (i+1)\rho(i)}\,
+ \eta(N,K)
\end{align*}
\begin{align*}
&\le  8C_2\Vert f \Vert_{\infty} \sigma_{N}^{-2}\frac{K+1}{\sqrt{N}} \sum_{i = 0}^{\infty} \rho(i) 
+ 2C_2\sqrt{2/\pi}\sigma_{N}^{-1}  \sum_{i= K+1}^{\infty}  \rho(i)  
\\
&\quad+ 9\max\{C_2,\sqrt{C_4}\} \sigma_{N}^{-1}\frac{K+1}{\sqrt{N}}\sqrt{\sum_{i=0}^{\infty} (i+1)\rho(i)}\, 
+ 2\max\{1,\sigma_{N}^{-2}\}\sqrt N\tilde\rho(K)
\\
&\le  \left(8\sqrt{2}C_2\Vert f \Vert_{\infty} \sigma_{N}^{-2} + 9\max\{C_2,\sqrt{C_4}\} \sigma_{N}^{-1}\right) \frac{K+1}{\sqrt{N}}\sqrt{\sum_{i=0}^{\infty} (i+1)\rho(i)}
 \\
&\quad+ 2C_2\sqrt{2/\pi}\sigma_{N}^{-1}  \sum_{i= K+1}^{\infty}  \rho(i)  
+ 2\max\{1,\sigma_{N}^{-2}\}\sqrt N\tilde\rho(K)
\\
& \le C_{\#}\!\left(\frac{K+1}{\sqrt{N}} + \sum_{i= K+1}^{\infty}  \rho(i) \right) + C_{\#}'\sqrt N\tilde\rho(K),
\end{align*}
where
\beqn
C_{\#} = 12 \max\{\sigma_{N}^{-1},\sigma_{N}^{-2}\} \max\{C_2,\sqrt{C_4}\} (1 + \Vert f \Vert_{\infty}) \sqrt{\sum_{i = 0}^{\infty}(i+1) \rho(i)}
\eeqn
and
\beqn
C_{\#}' = 2\max\{1,\sigma_{N}^{-2}\}.
\eeqn
 This completes the proof of Theorem~\ref{main thm 1d}. \qed

\newpage
\bigskip
\bigskip
\bibliography{CLT_with_a_rate_of_convergence}{}
\bibliographystyle{plainurl}


\vspace*{\fill}

\end{document}